\documentclass[leqno,10pt]{amsart}

\usepackage{amssymb}  
\usepackage{amsmath}
\usepackage[all]{xy}

\newcommand{\incl}[1][r]{\ar@<-0.2pc>@{^(-}[#1] \ar@<+0.2pc>@{-}[#1]}
\newcommand{\ZZZ}{\mathcal{Z}}      
\newcommand{\PPP}{\mathcal{P}} 
\newcommand{\RRR}{\mathcal{R}} 
\newcommand{\XXX}{\mathcal{X}}     
\newcommand{\YYY}{\mathcal{Y}}      
\newcommand{\NNN}{\mathcal{N}}
\newcommand{\FFF}{\mathcal{F}}      
\newcommand{\III}{\mathcal{I}}       
\newcommand{\OOO}{\mathcal{O}}      
\newcommand{\PPr}{\mathbb{P}}      
\newcommand{\HHH}{\mathcal{H}}

\newcommand{\HHHp}{\mathcal{H}^{\mathrm{main}}}
\newcommand{\BBB}{\mathcal{B}}      
\newcommand{\MMM}{\mathcal{M}}     

\newcommand{\NNe}{\mathbb{N}}

\newcommand{\Gmm}{\mathbb{G}_m}      
\newcommand{\Aff}{{\mathbb{A}}_{k}^{1}}

\renewcommand{\ggg}{\mathfrak{g}}      
\newcommand{\ssl}{\mathfrak{sl}}       
\newcommand{\rank}{\mathrm{rk}}    
\newcommand{\Irr}{\mathrm{Irr}}

\newcommand{\Gra}{\mathrm{Gr}}      
\newcommand{\Hom}{\mathrm{Hom}}

\newcommand{\tra}{\mathrm{tr}}          
\newcommand{\Ker}{\mathrm{Ker}}

\newcommand{\Aut}{\mathrm{Aut}}
\newcommand{\Stab}{\mathrm{Stab}} 
\newcommand{\End}{\mathrm{End}} 
\newcommand{\Ima}{\mathrm{Im}} 
\newcommand{\Mor}{\mathrm{Mor}}
\newcommand{\Hilb}{\mathrm{Hilb}}
\newcommand{\Spec}{\mathrm{Spec}}
\newcommand{\multi}{\mathrm{multi}}
\renewcommand{\labelitemi}{$\bullet$}

\theoremstyle{plain}
\newtheorem{theorem}{Theorem}[section]
\newtheorem{lemma}[theorem]{Lemma}
\newtheorem{proposition}[theorem]{Proposition}
\newtheorem*{lemma*}{Key-Proposition}
\newtheorem*{proposition*}{Reduction Principle}
\newtheorem{corollary}[theorem]{Corollary}
\newtheorem*{question}{Question}
\newtheorem*{maintheorem}{Theorem}

\theoremstyle{definition}
\newtheorem{definition}[theorem]{Definition}
\theoremstyle{remark}
\newtheorem{remark}[theorem]{Remark}

\begin{document}

\title[IHS and desingularizations of quotients]{Invariant Hilbert schemes and desingularizations of quotients by classical groups}

\author{Ronan TERPEREAU}



\begin{abstract}
Let $W$ be a finite-dimensional representation of a reductive algebraic group $G$. The invariant Hilbert scheme $\HHH$ is a moduli space that classifies the $G$-stable closed subschemes $Z$ of $W$ such that the affine algebra $k[Z]$ is the direct sum of simple $G$-modules with prescribed multiplicities. In this article, we consider the case where $G$ is a classical group acting on a classical representation $W$ and $k[Z]$ is isomorphic to the regular representation of $G$ as a $G$-module. We obtain families of examples where $\HHH$ is a smooth variety, and thus for which the Hilbert-Chow morphism $\gamma:\ \HHH \rightarrow W/\!/G$ is a canonical desingularization of the categorical quotient. 
\end{abstract}

\maketitle

\section*{Introduction and statement of the main results} \label{introo}

The main motivation for this article comes from a classical construction of canonical desingularizations of quotient varieties. Let $k$ be an algebraically closed field of characteristic zero, $G$ a reductive algebraic group over $k$, and $W$ a finite-dimensional linear representation of $G$. We denote 
$$\nu:\ W \rightarrow W/\!/G$$
the quotient morphism, where $W/\!/G:=\Spec(k[W]^G)$ is the categorical quotient. In general, $\nu$ is not flat and the variety $W/\!/G$ is singular. A universal "flattening" of $\nu$ is given by the invariant Hilbert scheme constructed by Alexeev and Brion (\cite{AB, Br}). 
We recall briefly the definition (see Section \ref{generaliteesHilbert} for details). Let $\Irr(G)$ be the set of isomorphism classes of irreducible representations of $G$, and $h$ a function from $\Irr(G)$ to $\NNe$. Such a function $h$ is called a \textit{Hilbert function}. The invariant Hilbert scheme $\Hilb_{h}^{G}(W)$ parametrizes the $G$-stable closed subschemes $Z$ of $W$ such that
$$k[Z] \cong \bigoplus_{M \in \Irr(G)} M^{\oplus h(M)}$$ 
as $G$-modules. If $h=h_W$ is the Hilbert function of the general fibers of $\nu$, then we denote
$$\HHH:=\Hilb_{h_W}^G(W)$$ 
to simplify the notation. Let $\XXX \subset \HHH \times W$, equipped with the first projection $\pi:\ \XXX \rightarrow \HHH$, be the \textit{universal family} over $\HHH$. The morphism $\pi$ is flat by definition and there is a commutative diagram
$$\xymatrix{
    \XXX \ar[r]^{\pi}  \ar[d]_{p}   & \HHH  \ar[d]^{\gamma}  \\
      W  \ar[r]_(0.4){\nu}              &  W/\!/G 
    }$$
where $p$ is the second projection and $\gamma$ is the \textit{Hilbert-Chow morphism} that sends a closed subscheme $Z \subset W$ to the point $Z/\!/G \subset W/\!/G$. The Hilbert-Chow morphism is projective and induces an isomorphism over the largest open subset $(W/\!/G)_* \subset W/\!/G$ over which $\nu$ is flat. In particular, $\nu$ is flat if and only if $\HHH \cong W/\!/G$. The \textit{main component} of $\HHH$ is the subvariety defined by
$$\HHHp:=\overline{\gamma^{-1}((W/\!/G)_*)}.$$   
Then the restriction 
$$\gamma:\ \HHHp \rightarrow W/\!/G$$
is a projective birational morphism.

\begin{question}
In which cases is the Hilbert-Chow morphism, possibly restricted to the main component, a desingularization of $W/\!/G$?
\end{question}

When $G$ is a finite group, Ito and Nakamura defined the $G$-Hilbert scheme $G$-$\Hilb(W)$ as the closure of the set of the free $G$-orbits in the fixed point subscheme of the punctual Hilbert scheme of $|G|$ points in $W$ (see \cite{IN1,IN2}). In particular, if $G$ acts freely on $W$, then the $G$-Hilbert scheme $G$-$\Hilb(W)$ coincides with the main component $\HHHp$. The case where $G$ is a finite subgroup of $SL_n$ acting on the defining representation $k^n$ has been extensively studied; let us recall some relevant results:
\begin{itemize}
\item If $\dim(W)=2$, then $\HHH$ is always a smooth variety. In particular, if $W=k^2$ and if $G \subset SL(W)$, then Ito and Nakamura showed that $\gamma$ is the minimal desingularization of the quotient surface $W/G$ (see [loc. cit.]). 
\item If $\dim(W)=3$ and $G \subset SL(W)$, Bridgeland, King and Reid showed, using homological methods, that once again $\HHH$ is a smooth variety and that $\gamma$ is a crepant desingularization of $W/G$ (see \cite{BKR}).
\item If $\dim(W)=4$ and $G \subset SL(W)$, then $\HHH$ can be singular. For instance, if $G \subset SL_2$ is the binary tetrahedral group and if $W$ is the direct sum of two copies of the defining representation, then Lehn and Sorger showed that $\HHH$ has two irreducible components but that $\HHHp$ is smooth (see \cite{LS}). 
\end{itemize} 
However, when $G$ is infinite, this question is open and completely unexplored. 
In this article, we study the case where $G$ is a classical group and $W$ is a classical representation of $G$. We then show that the invariant Hilbert scheme $\HHH$ is still a desingularization of $W/\!/G$ in "small" cases (but not in general). In addition, unlike the case where $G$ is finite, it may happen that $W/\!/G$ is smooth but $\nu$ is not flat; then $\gamma$ is not an isomorphism.

Let $V$, $V'$, $V_1$ and $V_2$ be vector spaces of dimension $n$, $n'$, $n_1$ and $n_2$ respectively. We consider the following cases:
\begin{enumerate} 
\item $G=SL(V)$ acting naturally on $W:=\Hom(V',V)=V^{\oplus n'}$, the direct sum of $n'$ copies of the defining representation;
\item $G=O(V)$ acting naturally on $W:=\Hom(V',V)=V^{\oplus n'}$; 
\item $G=Sp(V)$, with $n$ even, acting naturally on $W:=\Hom(V',V)=V^{\oplus n'}$; 
\item $G=GL(V)$ acting naturally on $W:=\Hom(V_1,V) \oplus \Hom(V,V_2)=V^{\oplus n_1} \oplus V^{* \oplus n_2}$, the direct sum of $n_1$ copies of the defining representation and $n_2$ copies of its dual.
\end{enumerate}

In these four cases, the description of the quotient morphism $\nu$ is well-known and follows from the \textit{First Fundamental Theorem} for the classical groups (\cite[\S 9.1.4,\S 11.1.2,\S 11.2.1]{Pr}): 
\begin{itemize}
\item \textit{Case 1.} $\nu$ is the natural map  
 $$\begin{array}{ccccc}
   \Hom(V',V)  & \rightarrow  &  \Hom(\Lambda^n(V'),\Lambda^n(V)) &\cong &\Lambda^n (V'^*), \\
         w          & \mapsto      &    \Lambda^n(w)& &  
\end{array}$$ 
where $\Lambda^n (V'^*)$ denotes the $n$-th exterior power of $V'^*$. We distinguish between three different cases:\\
\indent \ -if $n'<n$, then $W/\!/G=\{0\}$ and $\nu$ is trivial; \\
\indent \ -if $n'=n$, then $W/\!/G=\Lambda^n (V'^*) \cong \Aff$ and $\nu(w)=\det(w)$;\\
\indent \ -if $n' >n$, then $W/\!/G=C(\Gra(n,V'^*))$ is the affine cone over the Grassmannian $\Gra(n,V'^*)$ of the $n$-dimensional subspaces of $V'^*$ viewed as a subvariety of $\PPr(\Lambda^n V'^*)$ via the Pl\"{u}cker embedding.\\
One may check that $W/\!/G=\Lambda^n (V'^*)$ if and only if $n=1$ or $n \geq n'-1$. We will see in Section \ref{casSln} that $\nu$ is flat if and only if $n=1$ or $n' \leq n$. 
\item \textit{Case 2.} $\nu$ is the composite map 
$$\begin{array}{ccccc}
\Hom(V',V)   & \rightarrow  & \Hom(S^2(V'),S^2(V)) & \rightarrow & S^2(V'^*), \\
       w  & \mapsto      &  S^2(w)  & \mapsto &  q_1(S^2(w))
\end{array}$$ 
where $S^2(V'^*)$ denotes the symmetric square of $V'^*$, and $q_1$ is the morphism induced by the linear projection from $S^2(V)$ onto the line generated by a non-degenerate quadratic form. It follows that  
$$W/\!/G=S^2(V'^*)^{\leq n}:=\left\{ Q \in  S^2(V'^*) \  \mid \ \rank(Q) \leq n \right\} $$
is a symmetric determinantal variety. One may check (see \cite[Corollaire 3.1.7]{Te1}) that $\nu$ is flat if and only if $n \geq 2n'-1$. 
\item \textit{Case 3.} $\nu$ is the composite map 
$$\begin{array}{ccccc}
\Hom(V',V)   & \rightarrow  & \Hom(\Lambda^2(V'),\Lambda^2(V)) & \rightarrow & \Lambda^2(V'^*),  \\
        w  & \mapsto      &   \Lambda^2(w)  & \mapsto &  q_2(\Lambda^2(w))
\end{array}$$ 
where $q_2$ is the morphism induced by the linear projection from $\Lambda^2(V)$ onto the line generated by a non-degenerate skew-symmetric bilinear form. It follows that  
$$W/\!/G=\Lambda^2(V'^*)^{\leq n}:=\left\{ Q \in  \Lambda^2(V'^*) \  \mid \ \rank(Q) \leq n \right\} $$
is a skew-symmetric determinantal variety. One may check (see \cite[Corollaire 3.3.8]{Te1}) that $\nu$ is flat if and only if $n \geq 2n'-2$. 
\item \textit{Case 4.} $\nu$ is the natural map
$$\begin{array}{ccc}
 {\Hom}(V_1,V) \times {\Hom}(V,V_2)  & \rightarrow  & {\Hom}(V_1,V_2),  \\
        (u_1,u_2)  & \mapsto      &  u_2 \circ u_1
\end{array}$$
and thus
$$W/\!/G={\Hom}(V_1,V_2)^{\leq n}:=\{ f \in  {\Hom}(V_1,V_2) \ \mid  \ \rank(f) \leq n \}$$
is a determinantal variety. We will see in Section \ref{GLngeneral} that $\nu$ is flat if and only if $n \geq n_1+n_2-1$. 
\end{itemize}

The main result of this article is the following 

\begin{maintheorem} 
In the following cases, the invariant Hilbert scheme $\HHH$ is a smooth variety and the Hilbert-Chow morphism is the succession of blows-up described as follows:
\begin{itemize}
 \item \emph{Case 1}. Let $C_0$ be the blow-up of the affine cone $C(\Gra(n,V'^*))$ at 0.\\ 
If $n'>n>1$, then $\HHH$ is isomorphic to $C_0$. 
\item \emph{Case 2}. Let $\YYY_{0}^{s}$ be the blow-up of the symmetric determinantal variety $S^2(V'^*)^{\leq n}$ at 0.
\begin{itemize} 
\item If $n'>n=1$ or $n'=n=2$, then $\HHH$ is isomorphic to $\YYY_{0}^{s}$.
\item If $n'>n=2$, then $\HHH$ is isomorphic to the blow-up of $\YYY_{0}^{s}$ along the strict transform of $S^2(V'^*)^{\leq 1}$.
\end{itemize}
\item \emph{Case 3}. Let $\YYY_{0}^{a}$ be the blow-up of the skew-symmetric determinantal variety $\Lambda^2(V'^*)^{\leq n}$ at 0.
\begin{itemize} 
\item If $n'=n=4$, then $\HHH$ is isomorphic to $\YYY_{0}^{a}$.
\item If $n'>n=4$, then $\HHH$ is isomorphic to the blow-up of $\YYY_{0}^{a}$ along the strict transform of $\Lambda^2(V'^*)^{\leq 2}$.
\end{itemize}
\item \emph{Case 4}. Let $\YYY_0$ be the blow-up of the determinantal variety $\Hom(V_1,V_2)^{\leq n}$ at 0.
\begin{itemize}
\item If $\max(n_1,n_2)>n=1$ or $n_1=n_2=n=2$, then $\HHH$ is isomorphic to $\YYY_0$.  
\item If $\min(n_1,n_2) \geq n=2$ and $\max(n_1,n_2)>2$, then $\HHH$ is isomorphic to the blow-up of $\YYY_0$ along the strict transform of $\Hom(V_1,V_2)^{\leq 1}$.
\end{itemize}
\end{itemize}
\end{maintheorem}

When $W/\!/G$ is singular and Gorenstein, we will see that the desingularization $\gamma$ is never crepant. Also, we conjecture that in Cases 1-4, the invariant Hilbert scheme $\HHH$ is smooth if and only if $\nu$ is flat or we are in one of the cases of the above theorem. In this direction, we will show in another article (also partially extracted from \cite{Te1}) that $\HHH$ is singular in Cases 2 and 4 for $n=3$, and also in the case where $G=SO(V)$ acts naturally on $W=\Hom(V',V)$ with $n'=n=3$. 

A key ingredient in the proof of the main Theorem is a group action on $\HHH$ with finitely many orbits. Indeed, for any reductive algebraic group $G$, any finite dimensional $G$-module $W$, and any algebraic subgroup
$$ G' \subset \Aut^G(W),$$
it is known that $G'$ acts on $W/\!/G$ and $\HHH$, and that the quotient morphism and the Hilbert-Chow morphism are $G'$-equivariant. To describe the flat locus of $\nu$, it is almost enough to know the dimension of the fiber of $\nu$ over one point of each orbit. In the same way, determining the tangent space of $\HHH$ at a point of each closed orbit is enough to show that $\HHH$ is smooth, thanks to a semicontinuity argument. 

Another important ingredient of this article, which was already used by Becker in \cite[\S 4.1]{Be}, is the

\begin{lemma*} 
Let $G$, $W$ and $G'$ be as above. For any $M \in \Irr(G)$, there exists a finite-dimensional $G'$-submodule $F_M \subset \Hom^G(M,k[W])$ 
that generates $\Hom^G(M,k[W])$ as a $k[W]^G,G'$-module, and there exists
a $G'$-equivariant morphism
$${\delta}_{M}:\ \HHH \rightarrow \Gra(h_W(M),F_{M}^{*}).$$  
\end{lemma*} 

Thanks to the Key-Proposition, we obtain the following result which shows that it suffices to describe $\HHH$ in "small" cases to understand all cases:

\begin{proposition*}
Let $G$ and $W$ be as in Cases 1-4. We suppose that $n' \geq n$ resp. $n_1,n_2 \geq n$, and we fix $E \in \Gra(n,V'^*)$ resp. $(E_1,E_2) \in \Gra(n,V_1^*) \times \Gra(n,V_2)$. In Cases 1-3, we denote $W':=\Hom(V'/E^{\perp},V)$ resp. in Case 4, we denote $W':=\Hom(V_1/E_1^{\perp},V) \times \Hom(V,E_2)$, where $E^{\perp}$ resp. $E_1^{\perp}$, denotes the orthogonal subspace to $E$ in $V'$ resp. to $E_1$ in $V_1$.\\
Then the invariant Hilbert scheme $\HHH$ is the total space of a homogeneous bundle over $\Gra(n,V'^*)$ in Cases 1-3 resp. over $\Gra(n,V_1^*) \times \Gra(n,V_2)$ in Case 4, whose fiber is isomorphic to $\HHH':=\Hilb_{h_{W'}}^{G}(W')$.  
\end{proposition*}

For instance, to treat the case of $GL_2$ acting on $V^{\oplus n_1} \oplus V^{* \oplus n_2}$ with $n_1,n_2 \geq 2$, we just have to consider $V^{\oplus 2} \oplus V^{* \oplus 2}$. The reduction principle is the most important theoretical result of this article and will certainly be helpful to determine further examples of invariant Hilbert schemes.

To show the main Theorem, we have to proceed case by case but we follow a general method:
First, we perform the reduction step. Then, we look for the closed $G'$-orbits in $\HHH$, where $G'$ is a reductive algebraic subgroup of $\Aut^G(W)$. In Cases 1-4, such orbits are projective and thus contain fixed-points for the action of a Borel subgroup $B' \subset G'$. To determine these fixed-points, we significantly use representation theory of $B'$ and $G$. In Cases 1-4, we show that $\HHH$ has only one fixed-point that we denote $[Z_0]$. We deduce from Lemma \ref{fixespoints} that $\HHH$ is connected and that $[Z_0]$ belongs to the main component $\HHHp$. We then determine the Zariski tangent space $T_{[Z_0]} \HHH$ and we check that its dimension is the same as that of $\HHHp$. We thus get that $\HHH=\HHHp$ is a smooth variety.  

It is known that there exists a finite subset $\mathcal{E}$ of $\Irr(G)$ such that the morphism
$$\gamma \times \prod_{M \in \mathcal{E}} \delta_M:\ \HHH \longrightarrow W/\!/G \times \prod_{M \in \mathcal{E}} \Gra(h_W(M),F_M^*)$$
is a closed embedding; this is a consequence of the construction of the invariant Hilbert scheme as a closed subscheme of the multigraded Hilbert scheme of Haiman and Sturmfels (\cite{HS}). This suggests to choose an appropriate simple representation $M_1 \in \Irr(G)$ and to check whether $\gamma \times \delta_{M_1}$ is a closed embedding of $\HHH$. If this holds, then we have to identify the image; otherwise, we choose another simple representation $M_2$ and we look if $\gamma \times \delta_{M_1} \times \delta_{M_2}$ is a closed embedding. This procedure must stop after a finite number of steps and we get an explicit closed embedding of $\HHH$ as simple as possible.

In Section \ref{generaliteesHilbert}, we recall some basic results and we give a proof of the Key-Proposition. Case 1 (the easiest one) is treated in Section \ref{casSln}. Case 4 (the most difficult one) is treated for $n=2$ in Section \ref{GLngeneral}. The details for the other cases (except Case 2 for $n=1$ which is an easy exercise left to the reader!) can be found in the thesis \cite{Te1} from which this article is extracted.  

To conclude, let us mention that we can also use invariant Hilbert schemes to construct canonical desingularizations of some symplectic varieties. Let $G \subset GL(V)$ be as in Cases 2-4 and $W=V^{\oplus n'} \oplus V^{* \oplus n'}$. Then $W$ is a symplectic representation of $G$ and one can define a moment map $\mu:\ W \rightarrow \ggg^*$, where $\ggg$ is the Lie algebra of $G$. The \textit{symplectic reduction} of $W$ is defined as $\mu^{-1}(0)/\!/G$. In the article \cite{Te2} (also extracted from \cite{Te1}), we obtain families of examples for which $\gamma:\ \HHHp \rightarrow \mu^{-1}(0)/\!/G$ is a desingularization, sometimes symplectic, but where $\HHH$ is reducible in general.\\

\noindent \textbf{Acknowledgments:} I am very thankful to Michel Brion for proposing this subject to me and for a lot of helpful discussions, I thank Hanspeter Kraft for ideas and corrections concerning Section \ref{subsectionJ}, and I thank Tanja Becker for very helpful discussions during her stay in Grenoble in October 2010. I also thank the referees for carefully reading this paper.

\section{Generalities on invariant Hilbert schemes}  \label{generaliteesHilbert}

\subsection{}
The survey \cite{Br} gives a detailed introduction to the invariant Hilbert schemes. In this section, we recall some definitions and useful properties of these schemes. All the schemes we consider are supposed to be separated and of finite type over $k$. Let $G$ be a reductive algebraic group and $N$ a rational $G$-module, we have the decomposition
\begin{equation*} 
N \cong \bigoplus_{M \in \Irr(G)} N_{(M)} \otimes M,
\end{equation*}
where $N_{(M)}:=\Hom^G(M,N)$ is the space of $G$-equivariant morphisms from $M$ to $N$. The $G$-module $N_{(M)}\otimes M$ is called the isotypic component of $N$ associated to $M$ and $\dim(N_{(M)})$ is the multiplicity of $M$ in $N$. If, for any $M \in \Irr(G)$, we have $\dim(N_{(M)})< \infty$, then we define
\begin{equation*}
\begin{array}{ccccc}
h & : & \Irr(G) & \rightarrow & \NNe \\
& & M & \mapsto & \dim(N_{(M)}) 
\end{array}
\end{equation*}
the \textit{Hilbert function} of $N$.

Let $S$ be a scheme, $\ZZZ$ a $G$-scheme and $\pi: \ZZZ \rightarrow S$ an affine morphism, of finite type and $G$-invariant. According to \cite[\S 2.3]{Br}, the sheaf $\FFF:=\pi_{*} \OOO_{\ZZZ}$ admits the following decomposition as a (sheaf of) $\OOO_S,G$-modules: 
\begin{equation} \label{eq1}
\FFF \cong \bigoplus_{M \in \Irr(G)} \FFF_{(M)} \otimes M.
\end{equation}
The action of $G$ on $\FFF$ comes from the action of $G$ on each $M$, and each $\FFF_{(M)}:=\Hom^{G}(M,\FFF)$ is a coherent $\FFF^G$-module. From now on, we suppose that the family $\pi$ is \textit{multiplicity finite}, that is to say, $\FFF^G$ is a coherent $\OOO_S$-module. 
If, in addition, $\pi$ is flat, then each $\OOO_S$-module $\FFF_{(M)}$ is locally free of finite rank, and this rank is constant over each connected component of $S$. Let $h:\ \Irr(G) \rightarrow \NNe$ be a Hilbert function and $W$ a finite-dimensional $G$-module.

\begin{definition}
We define the Hilbert functor ${Hilb}_{h}^{G}(W)$: ${Sch}^{op} \rightarrow Sets$ by 
$$S \mapsto \left\{ 
\begin{array}{c} 
\xymatrix{
\ZZZ \ar@{.>}_{\pi}[dr]  \ar@{^{(}->}[r] &   S \times W \ar@{->}[d]^{p_1}\\
                     &  S } \end{array}
\middle| 
\begin{array}{l} \ZZZ \ \text{is a $G$-stable closed subscheme};  \\
{\pi}\ \text{is a flat morphism};\\
{\pi}_{*} {\OOO}_{\ZZZ} \cong \bigoplus_{M \in \Irr{(G)}} {\FFF}_{M} {\otimes} M;\\
 {\FFF}_{M} \ \text{is locally free of rank $h(M)$ over } {\OOO}_S.  \end{array} 
\right\} .$$
\end{definition} 

An element $(\pi:\ \ZZZ \rightarrow S) \in {Hilb}_{h}^{G}(W)(S)$ is called a \textit{flat family of $G$-stable closed subschemes of $W$ over $S$}. By \cite[Theorem 2.11]{Br}, the functor ${Hilb}_{h}^{G}(W)$ is represented by a quasi-projective scheme ${\Hilb}_{h}^{{G}}(W)$: the \textit{invariant Hilbert scheme} associated to the $G$-module $W$ and to the Hilbert function $h$. We recall that we denote $\HHH:=\Hilb_{h_W}^{G}(W)$, where $h_W$ is the Hilbert function of the general fibers of $\nu :\ W \rightarrow W/\!/G$. We denote $\XXX \subset W \times \HHH$, equipped with the second projection $\pi:\ \XXX \rightarrow \HHH$, the \textit{universal family} over $\HHH$. 

\begin{proposition} {\bf (cf. \cite[Proposition 3.15]{Br})} \label{chow}
With the above notation, the diagram
\begin{equation} \label{diag1}
        \xymatrix{
     \XXX \ar[d]_{\pi} \ar[r]^{q} & W \ar@{->>}[d]^(0.4){\nu} \\
     \HHH \ar[r]_(0.4){\gamma} & W/\!/G 
    }
\end{equation}
commutes, where $\pi$ and $q$ are the natural projections.\\
Moreover, the pull-back of $\gamma$ to the flat locus of $\nu$ is an isomorphism.
\end{proposition}

We fix an algebraic subgroup $G' \subset {\Aut}^G(W)$. Then we have:  

\begin{proposition} {\bf (cf. \cite[Proposition 3.10]{Br})} \label{groupaction}
With the above notation, $G'$ acts on $\HHH$ and on $\XXX$ such that all the morphisms of Diagram (\ref{diag1}) are $G'$-equivariant.
\end{proposition}

\begin{remark} 
With the above notation, the morphism ${\OOO}_{\HHH} \rightarrow \FFF:={\pi}_{*}{\OOO}_{\XXX}$ is a morphism of $G'$-modules, and the ${\FFF}_{(M)}$ that appear in the decomposition (\ref{eq1}) are ${\OOO}_{S},G'$-modules.  
\end{remark}

We are now ready to show the Key-Proposition stated in the introduction:

\begin{proof}[Proof of the Key-Proposition]
We use the notation of Diagram (\ref{diag1}). The inclusion $\iota:\ \XXX \subset W \times_{W/\!/G} \HHH$ is $G' \times G$-equivariant and so $\iota$ induces a surjective morphism of ${\OOO}_{\HHH},G' \times G$-modules ${p_2}_{*} {\OOO}_{\HHH {\times}_{W/\!/G} W} \rightarrow \FFF:={\pi}_{*} {\OOO}_{\XXX}$. But ${p_2}_{*} {\OOO}_{\HHH {\times}_{W/\!/G} W}={\OOO}_{\HHH} {\otimes}_{k[W/\!/G]} k[W]$, where we recall that $k[W/\!/G]={k[W]}^{G}$. We can then consider the decomposition into ${\OOO}_{\HHH},G' \times G$-modules 
\begin{equation*}
{\OOO}_{\HHH} {\otimes}_{k[W/\!/G]} k[W]\cong \bigoplus_{M \in \Irr(G)} {\OOO}_{\HHH} {\otimes}_{k[W/\!/G]} {k[W]}_{(M)} {\otimes} M,
\end{equation*}
where the action of $G'$ on ${\OOO}_{\HHH} {\otimes}_{k[W/\!/G]} k[W]$ is induced by the action of $G'$ on $W$, and $G'$ acts trivially on $M$. For each $M \in \ \Irr(G)$, we deduce a surjective morphism of ${\OOO}_{\HHH},G'$-modules 
\begin{equation} \label{m1}
{\OOO}_{\HHH} {\otimes}_{k[W/\!/G]} {k[W]}_{(M)} \twoheadrightarrow \FFF_{(M)} . 
\end{equation}
It follows that the vector space ${k[W]}_{(M)}$ generates ${\FFF}_{(M)}={\Hom}^{G}(M,\FFF)$ as a ${\OOO}_{\HHH},G'$-module. 
The vector space ${k[W]}_{(M)}$ is generally infinite-dimensional, but ${k[W]}_{(M)}$ is a ${k[W]}^{G}$-module of finite type, and thus there exists a finite-dimensional $G'$-module $F_M$ that generates ${k[W]}_{(M)}$ as a ${k[W]}^{G}$-module: 
\begin{equation} \label{m2}
{k[W]}^{G} {\otimes} F_M \twoheadrightarrow {k[W]}_{(M)} .
\end{equation} 
We deduce from (\ref{m1}) and (\ref{m2}) a surjective morphism of ${\OOO}_{\HHH},G'$-modules 
\begin{equation}  \label{pacfib}
{\OOO}_{\HHH} {\otimes} F_M \twoheadrightarrow {\FFF}_{(M)},
\end{equation}
where we recall that ${\FFF}_{(M)}$ is a locally free ${{\OOO}_{\HHH}}$-module of rank $h_W(M)$.
By \cite[Exercise 6.18]{EH}, such a morphism gives a morphism of schemes 
$$\delta_M:\ \HHH \rightarrow \Gra(\dim (F_M)-h_W(M),F_{M}),$$ 
and one may check that $\delta$ is $G'$-equivariant. Finally, we identify
$$\Gra(\dim(F_M)-h_W(M),F_M) \cong \Gra(h_W(M),F_M^*),$$
which completes the proof. 
\end{proof}

\begin{remark} 
The Key-Proposition holds more generally if we consider a Hilbert function $h$ such that $h(V_0)=1$, where $V_0$ denotes the trivial representation.  
\end{remark}

We now obtain a set-theoretic description of the morphism ${\delta}_M$. We recall that, for any $G$-module $M$, we have the canonical isomorphisms
\begin{equation} \label{isocanonique}
\begin{array}{rccccc}
k{[W]}_{(M)}:=&{\Hom}^{G}(M,k[W]) &  \cong  & {(M^{*} \otimes k[W])}^{G} & \cong & \Mor^G(W,M^*) . \\
&(m \mapsto \phi(m) \, f) & &  \phi \otimes f  &  & (w \mapsto f(w) \, \phi)
\end{array}
\end{equation}
Via these isomorphisms, the elements of the $G'$-module $F_M \subset k[W]_{(M)}$ identify with $G'$equivariant morphisms from $W$ to $M^{*}$. The map $\delta_M$ is given by:
\begin{equation*}
\delta_M: \HHH \rightarrow \Gra(\dim (F_M)-h_W(M),F_M),\ [Z] \mapsto \Ker(f_Z), 
\end{equation*}
where 
\begin{equation} \label{descirption_explicite}
\begin{array}{ccccc} 
f_Z & : & F_M & \twoheadrightarrow & \FFF_{M,Z}  \\
& & q & \mapsto & q_{|Z} 
\end{array}
\end{equation}
is the surjective linear map obtained by passing to the fibers in (\ref{pacfib}).

\subsection{}  

We fix a Borel subgroup $B'\subset G'$. We will obtain a series of elementary results that will be useful in the next sections to show that $\HHH$ is a smooth variety in some cases.

\begin{lemma} \label{fixespoints}
We suppose that $W/\!/G$ has a unique closed $G'$-orbit and that this orbit is a point $x$. Then, each $G'$-stable closed subset of $\HHH$ contains at least one fixed-point for the action of the Borel subgroup $B'$. Moreover, if $\HHH$ has a unique fixed-point, then $\HHH$ is connected.
\end{lemma}

\begin{proof}
The Hilbert-Chow morphism $\gamma$ is projective and $G'$-equivariant, so the set-theoretic fiber $\gamma^{-1}(x)$ is a projective $G'$-variety. Let $C$ be a closed subset of $\HHH$, then $\gamma(C)$ is a $G'$-stable closed subset of $W/\!/G$. Hence $x \in \gamma(C)$, that is, $C \cap \gamma^{-1}(x)$ is non-empty. Therefore, the Borel fixed-point Theorem (\cite[Theorem 10.4]{Bo}) yields that $C \cap \gamma^{-1}(x)$ contains at least one fixed-point for the action of $B'$. As a consequence, each connected component of $\HHH$ contains at least one fixed-point, hence the last assertion of the lemma.  
\end{proof}

\begin{lemma} \label{Hlisse4}
We suppose, as in Lemma \ref{fixespoints}, that $W/\!/G$ has a unique closed orbit $x$ for the action of $G'$, and we denote $\HHH^{B'}$ the set of  fixed-points for the Borel subgroup $B'$. Then we have the equivalence
$$ \HHH=\HHHp \text{ is a smooth variety}  \Leftrightarrow  \left\{
    \begin{array}{l}
        \forall [Z] \in \HHH^{B'},\ \dim(T_{[Z]} \HHH)=\dim(\HHHp); \text{ and }\\
        \HHH \text{ is connected.}
    \end{array}
\right.$$
\end{lemma}

\begin{proof}
The direction $\Rightarrow$ is easy. Let us prove the other implication.
We denote $d:=\dim (\HHHp)$. The set $E:=\{[Z] \in \HHH(k) \ \mid \ \dim(T_{[Z]}\HHH)>d\}$ is a $G'$-stable closed subset of $\HHH(k)$. If $E$ is non-empty, then $E$ contains one fixed-point of $B'$ by Lemma \ref{fixespoints}. Let $[Z_0]$ be this fixed-point, then we have $\dim(T_{[Z_0]} \HHH)>d$, which contradicts our assumption. It follows that $E$ is empty, and thus $\HHH$ is a smooth variety. Since $\HHH$ is connected by assumption, $\HHH$ has to be irreducible, and thus $\HHH=\HHHp$.  
\end{proof}

Let $[Z] \in \HHH$ be a closed point, $I_Z$ the ideal of the closed subscheme $Z \subset W$, and $R:=k[W]/I_Z$ the algebra of global sections of the structure sheaf of $Z$.

\begin{proposition} {\bf (cf. \cite[Proposition 3.5]{Br})} \label{isoTangent}
With the above notation, there is a canonical isomorphism
$$T_{[Z]} \HHH \cong \Hom_{R}^{G}(I_Z/I_Z^2,R),$$
where $\Hom_{R}^{G}$ stands for the space of $R$-linear, $G$-equivariant maps.
\end{proposition}

Next, let $N_1$ be a $G$-submodule of $k[W]$ contained in $I_Z$ such that the natural morphism of $R,G$-modules $\delta:\ R \otimes N_1 \rightarrow I_Z/I_Z^2$ is surjective; and let $N_2$ be a $G$-submodule of $R \otimes N_1$ such that we have the exact sequence of $R,G$-modules 
\begin{equation} \label{complexeCornFlakes}
\begin{array}{cccccc}
R \otimes N_2  &  \stackrel{\rho}{\longrightarrow} & R \otimes N_1 & \stackrel{\delta}{\longrightarrow} & I_Z/I_Z^2 & \rightarrow 0, \\
  &&f \otimes 1& \mapsto & \overline{f} & 
\end{array}
\end{equation}
where we denote $\overline{f}$ the image of $f \in I_Z$ in $I_Z/I_Z^2$.

Applying the left exact contravariant functor ${\Hom}_R(\,.\,,R)$ to the exact sequence (\ref{complexeCornFlakes}) and taking the $G$-invariants, we get the exact sequence of finite-dimensional vector spaces
\begin{equation} \label{complexeCornFlakes2}
  \xymatrix{
    0 \ar[r] &{\Hom}_R^{G}(I_Z/I_Z^2,R) \ar[r]^{\delta^*} & {\Hom}_R^{G}(R {\otimes} N_1,R) \ar[r]^{\rho^*} \ar[d]^{\cong} & {\Hom}_R^{G}(R \otimes N_2,R) \ar[d]^{\cong} \\
        &   &  \Hom^G(N_1,R) & \Hom^G(N_2,R)
  }
\end{equation}

Therefore, we have $T_{[Z]} \HHH \cong \Ima(\delta^*) = \Ker(\rho^*)$. Moreover, if the ideal $I_Z$ is $B'$-stable, then we can choose $N_1$ and $N_2$ as $B' \times G$-modules such that all the morphisms of the exact sequence (\ref{complexeCornFlakes}) are morphisms of $R,B' \times G$-modules and all the morphisms of the exact sequence (\ref{complexeCornFlakes2}) are morphisms of $B'$-modules. 

\begin{lemma} \label{InegTang}
With the above notation, suppose that $R \cong k[G]$ as $G$-modules. Then, $\dim({\Hom}_R^G(I_Z/I_Z^2,R)) = \dim(N_1)- \rank(\rho^*)$. In particular, if $\delta$ is an isomorphism, then $\dim({\Hom}_R^G(I_Z/I_Z^2,R)) = \dim(N_1)$. 
\end{lemma}

\begin{proof}
We have
\begin{align*}
\dim({\Hom}_R^{G}(I_Z/I_Z^2,R)) &= \dim(\Hom_R^{G}(R \otimes N_1,R))- \rank(\rho^*)\\  
                            &= \dim(\Hom^{G}(N_1,R))- \rank(\rho^*)\\  
                            &= \dim(\Mor^G(G,N_1^*))- \rank(\rho^*) \text{ since } R \cong k[G],\\  
                            &= \dim(N_1)- \rank(\rho^*).
\end{align*} 
\end{proof}

\section{Case of $SL_n$ acting on $(k^n)^{\oplus n'}$}  \label{casSln}

\subsection{}
We denote $G:=SL(V)$, $G':=GL(V')$, and $W:=\Hom(V',V)$. Consider the action of $G' \times G$ on $W$ given by:
\begin{equation*}  
\forall w \in W,\ \forall (g',g) \in G' \times G,\ (g',g).w:=g  \circ w \circ g'^{-1} .
\end{equation*}

We recall that $W/\!/G$ was described in the introduction. We also recall that $C(\Gra(n,V'^*))$ denotes the affine cone over $\Gra(n,V'^*)$, and that $C_0$ denotes the blow-up of $C(\Gra(n,V'^*))$ at 0. The aim of this section is to show the main Theorem in Case 1. Specifically, we will show:

\begin{theorem} \label{SLLn}
If $n=1$ or $n'\leq n$, then $\HHH \cong W/\!/G$ and the Hilbert-Chow morphism $\gamma$ is an isomorphism.\\
If $n'>n>1$, then $\HHH \cong C_0$ and $\gamma$ is the blow-up of $C(\Gra(n,V'^*))$ at 0.\\
In all cases, $\HHH$ is  a smooth variety and thus, when $W/\!/G$ is singular, $\gamma$ is a desingularization.
\end{theorem}

The cases $n=1$ or $n'\leq n$ are easy and are treated by Corollary \ref{cas_faciles_SLn}. The case $n'>n>1$ is handled by Proposition \ref{iso_final2SLn}.

\subsection{Generic fiber and flat locus of the quotient morphism}

One can check that the variety $W/\!/G$ is smooth except when $1<n<n'-1$ in which case $W/\!/G$ has a unique singularity at 0. Moreover, $W/\!/G$ is always normal (\cite[\S 3.2, Th\'{e}or\`{e}me 2]{SB}) and Gorenstein (\cite[\S 4.4, Th\'{e}or\`{e}me 4]{SB}) because $G$ is semisimple. When $n' \geq n$, the variety $W/\!/G$ is the union of two $G'$-orbits: the origin and its complement, denoted $U$.

\begin{proposition} \label{ouvert_platitudeSLn}
Let $U$ be the open orbit defined above.\\
1) If $n' \geq n$, then the fiber of $\nu$ over a point of $U$ is isomorphic to $G$.\\
2) If $n'> n>1$, then $U$ is the flat locus of $\nu$; otherwise, $\nu$ is flat. 
\end{proposition}

\begin{proof}
The first assertion is a well-known fact and is easily checked. Let us show the second assertion. If $n'<n$, then $\nu$ is trivial, hence $\nu$ is flat. If $n'=n$, then $W/\!/G \cong \Aff$ and $\nu$ is the determinant, which is flat by \cite[Exercise 10.9]{Ha}.\\
We now suppose that $n'>n$. We know that $\nu$ is flat over a non-empty open subset of $W/\!/G$ (\cite[Theorem 14.4]{Ei}), hence $\nu$ is flat over $U$ by $G'$-homogeneity. Then, one can check that
$\dim(\nu^{-1}(0))=\dim(\{w \in W\ |\ \rank(w) \leq n-1 \})=(n'+1)(n-1).$
By 1), the dimension of the fiber of $\nu$ over a point of $U$ is $n^2-1$. The fibers of a flat morphism have all the same dimension; hence, if $n>1$, $U$ is the flat locus of $\nu$. Otherwise, $W/\!/G$ is smooth and all the fibers of $\nu$ have the same dimension, and thus $\nu$ is flat by \cite[Exercise 10.9]{Ha}. 
\end{proof}

\begin{corollary} \label{fctHilbSLn}
If $n' \geq n$, the Hilbert function $h_W$ of the general fibers of $\nu$ is given by:
$$\forall M \in \Irr(G),\ h_W(M)=\dim(M).$$
\end{corollary}

We deduce from Propositions \ref{chow} and \ref{ouvert_platitudeSLn}:

\begin{corollary} \label{cas_faciles_SLn}
The Hilbert-Chow morphism is an isomorphism if and only if:
\begin{itemize} \renewcommand{\labelitemi}{$\bullet$}
\item $n'<n$, and then $\HHH$ is a reduced point; or
\item $n'=n$, then $\HHH \cong \Aff$ and $\det:\ W \rightarrow \Aff$ is the universal family; or
\item $n'>n=1$, then $\HHH \cong V'^*$ and $Id:\ V'^* \rightarrow V'^*$ is the universal family.
\end{itemize}
\end{corollary}

It remains to consider the case $n'>n>1$. To do this, we will use the reduction principle for $SL_n$ that will allow us to reduce to $n'=n$.

\subsection{Reduction principle for $SL_n$}  \label{redSLn}

From now on, we suppose that $n' \geq n$. We are going to show in Case 1 the reduction principle stated in the introduction. Specifically, we will prove: 

\begin{proposition} \label{reduction1}
We suppose that $n' \geq n$ and let $P$ be the parabolic subgroup of $G'=GL(V')$ that preserves a $n$-dimensional subspace of $V'^*$. Then there is a $G'$-equivariant isomorphism 
$$\begin{array}{ccccc}
\psi & : & G' {\times}^{P} \Aff & \cong & \HHH,  \\
& & (g',x)P & \mapsto & g'.x 
\end{array}$$
where we denote $G' {\times}^{P} \Aff:=(G' \times \Aff)/\equiv$ with $(g',x) \equiv (g'p^{-1},\det(p^{-1})x).$ 
\end{proposition}

In the next lemma, we use classical invariant theory to determine a finite-dimensional $G'$-submodule of $k[W]_{(V^*)}:=\Hom^G(V^*,k[W])$ that generates this $k[W]^G$-module in order to apply next the Key-Proposition to $M=V^*$. The reference that we use for classical invariant theory is \cite{Pr}. 

\begin{lemma} \label{exi1SLn}
With the above notation, the $k[W]^G$-module $k[W]_{(V^*)}$ is generated by $\Hom^G(V^*,W^*)$.
\end{lemma}

\begin{proof}
We want to show that the natural map
$${k[W]}^G \otimes \Hom^G(V^*,W^*) \rightarrow  k[W]_{(V^*)}$$ 
of $k[W]^G,G'$-modules is surjective. We identify $k[W]$ with $S(W^*)$, the symmetric algebra of $W$. Then
\begin{align*}
  k[W]_{(V^*)} &\cong (S(W^*) \otimes V)^G \\
               &\cong \bigoplus_{p \geq 0} (S^p({(V^*)}^{n'})  \otimes V)^G\\
               &\cong \bigoplus_{p \geq 0}  \left( \bigoplus_{p_1+\ldots+p_{n'}=p}  (S^{p_1}(V^*) \otimes \cdots \otimes S^{p_{n'}}(V^*) \otimes V)^G \right ).
\end{align*} 
We fix $(p_1,\ldots,p_{n'}) \in \NNe^{n'}$. Then
$$(S^{p_1}(V^*) \otimes \cdots \otimes S^{p_{n'}}(V^*)  \otimes V)^G \cong (k[ V \oplus \cdots\oplus V \oplus V^*]_{(p_1,\ldots,p_{n'},1)})^G$$ 
is the vector space of multihomogeneous invariants of multidegree $(p_1,\ldots,p_{n'},1)$. \\
We apply the \textit{polarization operator} $\PPP$ defined in \cite[\S 3.2.1]{Pr}:
$$\PPP:\ (k[ V \oplus \cdots \oplus V \oplus V^{*}]_{(p_1,\ldots,p_{n'},1)})^G \longrightarrow (k[ {V}^{\oplus p_1} \oplus \cdots \oplus {V}^{\oplus p_{n'}} \oplus V^{*}]_{\multi})^G .$$ 
So we have to study the space of multilinear invariants
$$(k[ {V}^{\oplus p_1} \oplus \cdots \oplus {V}^{\oplus p_{n'}} \oplus V^{*}]_{\multi})^G \cong ({{V}^{*}}^{\otimes p_1} \otimes \cdots \otimes {{V}^{*}}^{\otimes p_{n'}} \otimes V)^G .$$
For each $1 \leq i \leq p$, where $p:=p_1+\ldots+p_{n'}$, we denote $\phi_i$ the vector corresponding to the $i$-th copy of $V^*$, and $v \in V$ the vector corresponding to the unique copy of $V$. According to the First Fundamental Theorem for $SL_n$ (\cite[\S 11.1.2]{Pr}), the vector space 
$({{V}^{*}}^{\otimes p_1} \otimes \cdots \otimes {{V}^{*}}^{\otimes p_{n'}} \otimes V)^G$ is generated by $({{V}^{*}}^{\otimes p_1} \otimes \cdots \otimes {{V}^{*}}^{\otimes p_{n'}})^G$ and by the bilinear forms $(\phi_j,v) \mapsto \phi_j(v)$, for $j=1,\ldots,p$.
Then, by \cite[\S 3.2.2, Theorem]{Pr}, the vector space 
$$(S^{p_1}(V^{*}) \otimes \cdots \otimes S^{p_{n'}}(V^{*}) \otimes V )^G$$
is generated by
$$\left\{  \RRR f,\ f \in ({V^*}^{\otimes p_1} \otimes \cdots \otimes {V^*}^{\otimes p_{n'}} \otimes V)^G \right\},$$ 
where we denote $\RRR$ the \textit{restitution operator} defined in \cite[\S 3.2.2]{Pr}. 
For $j=1,\ldots,n'$, we now denote $\phi_j \in V^*$ the vector corresponding to the $j$-th copy of $V^*$, and $v \in V$ as before. We have
$$\RRR f (\phi_1,\ldots,\phi_{n'},v)= \sum_{j=1}^{n'} \Psi_j (\phi_1,\ldots,\phi_{n'}) \times  \phi_j(v).$$ 
For all $1 \leq j \leq n'$, we have $\Psi_j \in k[W]^G$ and $\phi_j(v) \in ({W^* \otimes V })^G \cong \Hom^G(V^*,W^*)$, hence the result. 
\end{proof}

By (\ref{isocanonique}), the $G'$-module $V' \cong \Hom^G(V^*,W^*)$ identifies with the linear maps of $\Mor^G(W,V)$, that is, with $\Hom^G(W,V)$. In fact, this identification is given simply by sending a morphism to its transpose. Denoting $p_1, \ldots ,p_{n'}$ the $n'$ linear projections from $W \cong V'^* \otimes V$ to $V$, we will sometimes identify the $G'$-module $V'$ with $\left \langle p_1, \ldots ,p_{n'} \right \rangle = \Hom^G(W,V)$. With this notation, Lemma \ref{exi1SLn} admits the following reformulation: any $G$-equivariant morphism from $W$ to $V^*$ can be written as a linear combination of the form $\sum_i f_i p_i$, for some $f_i \in k[W]^G$.

By Corollary \ref{fctHilbSLn}, we have $h_W(V^*)=\dim(V^*)=n$, and thus the Key-Proposition applied for $M=V^*$ gives a $G'$-equivariant morphism $\HHH \rightarrow \Gra(n,V'^*)$. The Grassmannian $\Gra(n,V'^*)$ is a homogeneous space for the natural action of $G'$. Recalling that $n' \geq n$ by assumption, we denote $E:=\left \langle p_1,\ldots,p_n \right \rangle$, which is a $n$-dimensional subspace of $V'^*$, and $P \subset G'$ the stabilizer of $E$, which is a parabolic subgroup. Identifying $\Gra(n,V'^*) \cong G'/P$, we obtain a $G'$-equivariant morphism
\begin{equation}  \label{rro}
 \rho: \HHH \rightarrow G'/P,
\end{equation}   
whence a $G'$-equivariant isomorphism
\begin{equation*}
\HHH \cong G' \times^P F,
\end{equation*}
where $F$ is the scheme-theoretic fiber of $\rho$ over $eP$. Therefore, to show Proposition \ref{reduction1}, we just have to show: 

\begin{lemma} \label{fibrehil}
With the above notation, there is a $P$-equivariant isomorphism
$$ F \cong \Aff, $$
where the action of $P$ on $\Aff$ is given by $p.x:=\det(p^{-1})x$.
\end{lemma}

\begin{proof}
By definition of $F$, for any scheme $S$, we have 
 $$ \Mor(S,F) = \left\{ 
\begin{array}{c} 
\xymatrix{
\ZZZ \ar@{.>}_{\pi}[dr]  \ar@{^{(}->}[r] &   S \times W \ar@{->}[d]^{p_1}\\
                     &  S } \end{array}
\middle|
\begin{array}{l} {\ZZZ} \ \text{is a $G$-stable closed subscheme};  \\
{\pi} \ \text{is a flat morphism};\\
{\pi}_{*} {\OOO}_{{\ZZZ}} = \bigoplus_{M \in \Irr{G}} {\FFF}_{M} {\otimes} M;\\
 {\FFF}_{M} \ \text{is a loc. free ${\OOO}_S$-module of rank $h_W(M)$}; \\
\forall s \in S(k),\ \rho({{\ZZZ}}_s)=E. \end{array} 
\right\} . $$
But
\begin{align*}
\rho({{\ZZZ}}_s)=E &\Leftrightarrow  {{p}_{n+1}}_{|{{\ZZZ}}_s}= \cdots ={{p}_{n'}}_{|{{\ZZZ}}_s}=0\\
                          &\Leftrightarrow  {{\ZZZ}}_s \subset \Hom(V'/E^{\perp},V)=:W'.
\end{align*} 
Hence 
\begin{align*} \Mor(S,F) &= \left\{ 
\begin{array}{c} 
\xymatrix{
\ZZZ \ar@{.>}_{\pi}[dr]  \ar@{^{(}->}[r] &   S \times W' \ar@{->}[d]^{p_1}\\
                     &  S } \end{array}
\middle|
\begin{array}{l} {\ZZZ} \  \text{ is a $G$-stable closed subscheme};  \\
{\pi} \ \text{is a flat morphism};\\
{\pi}_{*} {\OOO}_{{\ZZZ}} = \bigoplus_{M \in \Irr{G}} {\FFF}_{M} {\otimes} M;\\
 {\FFF}_{M} \ \text{is a loc. free ${\OOO}_S$-module of rank $h_{W}(M)$.}   \end{array} 
\right\} \\
&={Hilb}_{h_{W}}^{{G}} (W')(S) \\
&\cong \Mor(S,{\Hilb}_{h_{W}}^{{G}} (W')),  
\end{align*} 
where the last isomorphism comes from the definition of the invariant Hilbert scheme in terms of a representable functor.\\
By Corollary \ref{fctHilbSLn}, we have $h_W=h_{W'}$. It follows that 
$$F \cong \Hilb_{h_{W'}}^{G}(W') \cong \Aff$$ as $P$-schemes, where the last isomorphism follows from Corollary \ref{cas_faciles_SLn}. 
\end{proof}

\begin{remark}
We have $G'$-equivariant morphisms 
\begin{equation*}
\XXX \stackrel{\pi}{\longrightarrow} \HHH \stackrel{\rho}{\longrightarrow} G'/P,
\end{equation*}
where $\pi:\ \XXX \rightarrow \HHH$ is the universal family. Using the same arguments than those used to show Proposition \ref{reduction1}, one can show that there exists a $G'$-equivariant isomorphism
\begin{equation} \label{uunnii}
\XXX \cong G' {\times}^{P} \Hom(V'/E^{\perp},V),
\end{equation} 
where the action of $P$ on $\Hom(V'/E^{\perp},V)$ is the natural one. In other words, if we denote $T$ resp. $\underline{V'}$ and $\underline{V}$, the tautological bundle resp. the trivial bundles with respective fibers $V'$ and $V$, over the Grassmannian $\Gra(n,V'^*)$, then $\XXX$ identifies with the total space of the vector bundle $\Hom(\underline{V'}/T^{\perp},\underline{V}) \cong T \otimes \underline{V}.$ 
\end{remark}

\subsection{The case $1<n<n'$}

From now on, we suppose that $1<n<n'$. In the preceding section, we showed that $\HHH$ is the total space of a homogeneous line bundle over $\Gra(n,V'^*)$. In this section, we will show that this line bundle is the tautological one, and thus that the Hilbert-Chow morphism $\gamma$ is the blow-up of $W/\!/G=C(\Gra(n,V'^*))$ at 0.

\begin{lemma} \label{morphisme_dans_YSLn}
Let $\rho$ be the morphism defined by (\ref{rro}) and $C_0$ the smooth variety obtained by blowing-up the affine cone $C(\Gra(n,V'^*))$ at 0. Then the morphism $\gamma \times \rho$ sends $\HHH$ into $C_0$.
\end{lemma}

\begin{proof}
By Propositions \ref{chow} and \ref{ouvert_platitudeSLn}, the restriction of $\gamma$ to $\gamma^{-1}(U)$ is an isomorphism. We fix $y_0 \in U$ and we denote $Q:=\Stab_{G'}(y_0)$ and $[Z_0]$ the unique point of $\HHH$ such that $\gamma([Z_0])=y_0$. As $\gamma$ is $G'$-equivariant, $[Z_0]$ has to be $Q$-stable. In addition, $\rho$ is also $G$-equivariant, hence $\rho([Z_0])$ is a $Q$-stable line in $W/\!/G$. But one may check that the unique $Q$-stable line of $W/\!/G$ is the one generated by $y_0$, hence $\gamma([Z_0]) \in \rho([Z_0])$. It follows that $(\gamma \times \rho)([Z_0]) \in C_0$. \\
As $C_0$ is a $G'$-stable closed subset of $W/\!/G \times \PPr(W/\!/G) $, we have $(\gamma \times \rho)([Z]) \in C_0$, for each $[Z] \in \gamma^{-1}(U)$, and thus $(\gamma \times \rho)^{-1}(C_0)$ is a closed subset of $\HHH$ containing $\gamma^{-1}(U)$. Hence $(\gamma \times \rho)^{-1}(C_0)=\HHH$ as expected.    
\end{proof}

We recall that we have the diagram 
\begin{equation}  \label{ddiiaa}
\xymatrix{ 
   C_0=\{(x,L) \in C(\Gra(n,V'^*)) \times \Gra(n,V'^*)\ |\ x \in L\} \ar[rd]^{p_2}  \ar[d]_{p_1}  &   \\
      C(\Gra(n,V'^*)) & \Gra(n,V'^*)  }
\end{equation}  
where we denote $p_1$ and $p_2$ the natural projections.

\begin{proposition} \label{iso_final2SLn}
With the same notation as in Lemma \ref{morphisme_dans_YSLn}, $\gamma \times \rho:\ \HHH \rightarrow C_0$ is an isomorphism.
\end{proposition}

\begin{proof}
We identify $\Gra(n,V'^*) \cong G'/P$ as in Section \ref{redSLn}. The projection $p_2:\ C_0 \rightarrow G'/P$ of Diagram (\ref{ddiiaa}) provides a $G'$-homogeneous line bundle structure to $C_0$ over $G'/P$. Let $D$ be the scheme-theoretic fiber of $p_2$ over $eP$, then $D \cong \Aff$ and $P$ acts on $\Aff$ by $p.x=\det(p^{-1})x$. It follows that there is a $G'$-equivariant isomorphism 
$$C_0 \cong G' \times^{P} \Aff.$$
Therefore, we have the commutative diagram
\begin{equation*} 
 \xymatrix{ 
   \HHH \ar[rr]^{\gamma \times \rho}   && C_0   \\
  {{G'} {\times}^{P} \Aff}  \ar@{.>}[rr]^{\theta}  \ar@{->>}[rd] \ar[u]^{\cong} && {{G'} {\times}^{P} \Aff} \ar@{->>}[ld] \ar[u]_{\cong} \\ & G'/{P} }  
  \end{equation*}
where $\theta$ denotes the $G'$-equivariant morphism such that the square commutes. We denote $\theta_e:\ \Aff \rightarrow \Aff$ the $P$-equivariant morphism obtained from $\theta$ by restriction to the fiber over $eP$. It follows from Lemma \ref{fibrehil} that the morphism $\theta_e$ identifies with the Hilbert-Chow morphism for $n'=n$ and thus, by Corollary \ref{cas_faciles_SLn}, $\theta_e$ is an isomorphism. Hence $\theta$ is an isomorphism as well as $\gamma \times \rho$. 
\end{proof}

\begin{proposition}  \label{res_crepantes_SLn}
We suppose that $W/\!/G$ is singular, that is, $1<n<n'-1$. Then, the desingularization $\gamma:\ \HHH \rightarrow W/\!/G$ is not crepant.
\end{proposition} 

\begin{proof}
We use the notation of Diagram (\ref{ddiiaa}) and we identify $\gamma:\HHH  \rightarrow W/\!/G$ with the blow-up $p_1:\ C_0 \rightarrow W/\!/G$ by Theorem \ref{SLLn}. We denote $\sigma:\ \Gra(n,V'^*) \rightarrow C_0$ the zero section. The exceptional divisor of $ p_1 $, denoted $D_{p_1} $, identifies with $\Gra(n,V'^*)$ via $\sigma$.
For every Cohen-Macaulay variety $X$, we denote $\omega_X$ its dualizing sheaf. By the Adjunction Formula (\cite[Proposition 8.20]{Ha}): 
$$\omega_{D_{p_1} } \cong \omega_{C_0} \otimes \OOO(D_{p_1}) \otimes \OOO_{D_{p_1}}.$$
By \cite[Proposition 6.18]{Ha}, we have $\OOO(-D_{p_1}) \cong \III_{D_{p_1}}$, where $\III_{D_{p_1}}$ is the sheaf ideal of $D_{p_1}$ in $C_0$. But $D_{p_1}$ is the image of the zero section in $C_0$, and thus ${\III_{D_{p_1}}}_{|D_{p_1}} \cong \OOO_{D_{p_1}}(1)$. It follows that
$$\omega_{D_{p_1}} \cong \omega_{C_0} \otimes \OOO_{D_{p_1}}(-1).$$
If ${p_1}$ is crepant, then 
$$\omega_{C_0} \cong p_{1}^{*}(\omega_{W/\!/G}) \cong p_{1}^{*}(\OOO_{W/\!/G}) \cong \OOO_{C_0},$$ 
where the second isomorphism follows from the well-known fact that, for a Gorenstein affine cone $X$, we have $\omega_X \cong \OOO_X$ (see \cite[\S A.1.1]{Te1} for a proof).\\
Hence $\omega_{D_{p_1}} \cong \OOO_{D_{p_1}}(-1)$. But $D_{p_1} \cong \Gra(n,V'^*)$, and thus 
$$\omega_{D_{p_1}} \cong \omega_{\Gra(n,V'^*)} \cong \OOO_{\Gra(n,V'^*)}(-n') \cong \OOO_{D_{p_1}}(-n').$$ 
As $n'>1$, we get a contradiction. It follows that ${p_1}$ cannot be crepant.  
\end{proof}

\section{Case of $GL_2$ acting on $(k^2)^{\oplus n_1} \oplus (k^{2*})^{\oplus n_2} $}  \label{GLngeneral}

\subsection{}  \label{yy}

We denote $G:=GL(V)$, $G':=GL(V_1) \times GL(V_2)$, and $W:=\Hom(V_1,V) \oplus \Hom(V,V_2)$. We consider the action of $G' \times G$ on $W$ given by:
\begin{equation*}  
\forall (u_1,u_2) \in W,\ \forall (g_1,g_2,g) \in G' \times G,\ (g_1,g_2,g).(u_1,u_2):=(g \circ u_1 \circ g_1^{-1},g_2 \circ u_2 \circ g^{-1}).
\end{equation*}  
We recall that $W/\!/G=\Hom(V_1,V_2)^{\leq n}$, and that we denote 
\begin{equation*} 
\YYY_0:=\left\{(f,L)\in \Hom(V_1,V_2)^{\leq 2} \times \PPr (\Hom(V_1,V_2)^{\leq 2})\ |\ f \in L \right\}=\OOO_{\PPr(\Hom(V_1,V_2)^{\leq 2})}(-1) 
\end{equation*}
the blow-up of $\Hom(V_1,V_2)^{\leq 2}$ at 0, and $\YYY_1$ the blow-up of $\YYY_0$ along the strict transform of $\Hom(V_1,V_2)^{\leq 1}$. The aim of this section is to show the main Theorem for Case 2. More precisely, we will show:  

\begin{theorem} \label{GL2}
If $n=2$ and $n_1+n_2 \leq 3$, then $\HHH \cong \Hom(V_1,V_2)$ and the Hilbert-Chow morphism $\gamma$ is an isomorphism.\\
If $n_1=n_2=n=2$, then $\HHH \cong \YYY_0$ and $\gamma$ is the blow-up of $\Hom(V_1,V_2)$ at 0.\\
If $\min(n_1,n_2) \geq n=2$ and $\max(n_1,n_2)>2$, then $\HHH \cong \YYY_1$ and $\gamma$ is the composition of blows-up that define $\YYY_1$.\\
In all cases, $\HHH$ is  a smooth variety and thus, when $W/\!/G$ is singular, $\gamma$ is a desingularization.
\end{theorem}

\begin{remark}
The case $\min(n_1,n_2)=1$ and $\max(n_1,n_2) \geq 3$ is quite different from the other cases and thus is not treated by Theorem \ref{GL2}. In this case, we show in \cite[\S 2.1.2]{Te1} that the main component $\HHHp$ is always smooth, and that $\HHH=\HHHp$ for $n=1$. We do not know if $\HHH=\HHHp$ for an arbitrary $n$.
\end{remark}

The cases $n=2$ and $n_1+n_2 \leq 3$ are easy and are treated by Corollary \ref{cas_facile}. The case $n_1=n_2=n=2$ is handled by Proposition \ref{gammaiso}. Finally, we obtain the last case from the previous one by using the reduction principle for $GL_n$ as we did for $SL_n$ in Section \ref{casSln}.

In Sections \ref{description_quotient} to \ref{reducGLn}, we do not make any assumption about the value of $n$. Indeed, even if we are interested in the case $n=2$, the results that we will show in those sections are quite general and they do not admit any real simplification if we suppose $n=2$.

\subsection{Generic fiber and flat locus of the quotient morphism} \label{description_quotient}

The most important results of this section are Propositions \ref{descriptiongeofib} and \ref{ouvertplatitude} that describe the fibers and the flat locus of the quotient morphism $\nu$. 

This morphism is also studied in \cite[\S II.4.1]{Kr}. However, this reference does not contain all the results that we will need subsequently and our formulations, notation and methods are somewhat different.

If $n_1,n_2 >n$, then $W/\!/G=\Hom(V_1,V_2)^{\leq n}$ is of dimension $nn_1+nn_2-n^2$, and by \cite[\S 6.1]{We}, $\Hom(V_1,V_2)^{\leq n}$ is normal, Cohen-Macaulay, and ${\Hom}(V_1,V_2)^{\leq n-1}$ is its singular locus. Moreover, by \cite[Theorem 5.5.6]{Sv}, the variety $W/\!/G$ is Gorenstein if and only if $n_1=n_2$. Otherwise, $W/\!/G=\Hom(V_1,V_2)$ is an affine space. We denote
$$N:=\min(n_1,n_2,n).$$ 
The action of $G'$ on $W$ induces the following action on $W/\!/G$:
\begin{equation*}
\forall (g_1,g_2) \in G',\  \forall f \in  {\Hom}(V_1,V_2),\ (g_1,g_2).f:=g_2 \circ f \circ g_{1}^{-1}. 
\end{equation*}
The variety $W/\!/G$ decomposes into $N+1$ orbits for this action:
\begin{equation*}
U_i:=\{ f \in  {\Hom}(V_1,V_2) \ \mid  \ \rank(f) = i \}
\end{equation*}
for $i=0, \ldots,N$; the closures of these orbits being nested in the following way
\begin{equation*}
\{0\}=\overline{U_0} \subset \overline{U_1} \subset \cdots\subset \overline{U_{N}}=W/\!/G.
\end{equation*}
We note that $U_N$ is the unique open orbit of $W/\!/G$.

\begin{definition} \label{defnilk}
The \textit{null cone of $\nu$}, denoted $\NNN(W,G)$, is the scheme-theoretic fiber of $\nu$ over $0$. 
\end{definition}

Some geometric properties of $\NNN(W,G)$ are obtained in \cite{KS}. For instance, $\NNN(W,G)$ is always reduced, but $\NNN(W,G)$ is irreducible if and only if $n_1+n_2 \leq n$ (\cite[Theorem 9.1]{KS}). We are going to determine the irreducible components of $\NNN(W,G)$ and their dimensions. We fix $m \in \{0, \ldots,n\}$ and we define the set 
\begin{equation*} 
X_m:=\left\{(u_1,u_2)\in W\ \middle| 
    \begin{array}{ll}
       \Ima(u_1) \subset \Ker(u_2);\\
       \rank(u_1) \leq \min(n_1,m);\\
       \dim(\Ker(u_2)) \geq \max(n-n_2,m). 
    \end{array}
 \right\}.
\end{equation*} 
We consider the diagram
\begin{equation*}
\xymatrix{ &  Z_m \ar@{->>}[ld]_{p_1} \ar@{->>}[rd]^{p_2} \\   X_m && \Gra(m,V) }
\end{equation*} 
where 
$$Z_m:=\{(u_1,u_2,L) \in \Hom(V_1,V) \times \Hom(V,V_2)\times \Gra(m,V)\ \mid  \ \Ima(u_1) \subset L \subset \Ker(u_2)\}$$ 
and the $p_i$ are the natural projections. We fix $L_0 \in \Gra(m,V)$. The second projection equips $Z_m$ with a structure of homogeneous vector bundle over $\Gra(m,V)$ whose fiber over $L_0$ is isomorphic to $F_m:=\Hom(V_1,L_0) \times \Hom(V/L_0,V_2)$. In other words, we have $Z_m=\Hom(\underline{V_1},T) \times \Hom(\underline{V}/T,\underline{V_2})$, where $T$ is the tautological bundle of $\Gra(m,V)$ and $\underline{V},\underline{V_1},\underline{V_2}$ are the trivial bundles with fibers $V,\ V_1$ and $V_2$ respectively. Hence $Z_m$ is a smooth variety of dimension
\begin{align*}
\dim(Z_m)&=\dim(\Gra(m,V))+\dim(\Hom(V_1,L_0) \times \Hom(V/L_0,V_2))\\
         &=m(n-m)+n_1 m+(n-m)n_2.
\end{align*}

\begin{proposition} \label{compirredfibzero}
With the above notation, each $X_m$ is an irreducible closed subset of $W$ and the irreducible components of the null cone $\NNN(W,G)$ are
$$\left\{
    \begin{array}{ll}
        X_i,\ \textit{for}\ i={\max(0,n-n_2)}, {\max(0,n-n_2)}+1, \ldots ,{\min(n,n_1),} &\text{ if } n<n_1+n_2;\\
        X_{n_1} &\text{ if } n\geq n_1+n_2.
    \end{array}
\right.$$
In addition, if $m \leq n_1$ or $m \geq n-n_2$, the map $p_1:\ Z_m \rightarrow X_m$ is birational. 
\end{proposition}

\begin{proof}
First, by definition, the $X_m$ are closed subsets of $W$. The morphism $p_1$ is surjective and $Z_m$ is irreducible, hence $X_m$ is irreducible. 
Then
$$\NNN(W,G)=\{(u_1,u_2)\in \Hom(V_1,V) \times \Hom(V,V_2)\ \mid  \ \Ima(u_1) \subset \Ker(u_2) \}= \bigcup_{i=0}^{n} X_i . $$
If  $n_1 \leq n-n_2$, then  $$\left\{
    \begin{array}{ll}
        X_0 \subset  \cdots \subset X_{n_1};   \\
        X_{n_1} = \cdots=X_{n-n_2};\\
        X_{n-n_2} \supset  \cdots \supset X_n;
    \end{array}
\right.$$  
and thus $X=X_{n_1}$.\\
If $n_1 > n-n_2$, then  $$\left\{
    \begin{array}{ll}
        X_0 \subset  \cdots \subset X_{\max(0,n-n_2)};   \\
        X_{\min(n,n_1)} \supset  \cdots \supset X_n;
    \end{array}
\right.$$  
and one easily checks that there is no other inclusion relation between the $X_m$. It remains to show the last assertion of the proposition. We define
$$Z'_m:=\{(u_1,u_2,L) \in Z_m\ \mid  \ \rank(u_1)=\min(m,n_1) \text{ and } \dim(\Ker(u_2))=\max(m,n-n_2)\}$$ and 
$$X'_m:=\{(u_1,u_2)\in X_m\ \mid  \ \rank(u_1)=\min(m,n_1) \text{ and } \dim(\Ker(u_2))=\max(m,n-n_2)\} .$$ 
It is clear that $Z'_m$ resp. $X'_m$, is a dense open subset of $Z_m$ resp. of $X_m$, and that we have $Z'_m=p_{1}^{-1}(X'_m)$.
If $m \leq n_1$ or $m \geq n-n_2$, then $p_1:\ Z'_m \rightarrow X'_m$ is an isomorphism, and thus $p_1$ is birational. 
\end{proof}

\begin{corollary} \label{fibzero}
The dimension of the null cone $\NNN(W,G)$ is:\\
$\bullet \ n n_2$ if $n \leq n_2-n_1$;\\
$\bullet \ n n_1$ if $n \leq n_1-n_2$;\\
$\bullet \ \frac{1}{4}n(n+2n_1+2n_2)+\frac{1}{4}{(n_1-n_2)}^2$ if $|n_1-n_2|<n<n_1+n_2$ and $n+n_1-n_2$ is even; \\
$\bullet \ \frac{1}{4}n(n+2n_1+2n_2)+\frac{1}{4}{(n_1-n_2)}^2-\frac{1}{4}$ if $|n_1-n_2|<n<n_1+n_2$ and $n+n_1-n_2$ is odd; \\ 
$\bullet \ n n_1+n n_2 -n_1 n_2$ if $n \geq n_1+n_2$.
\end{corollary}

\begin{proof}
By Proposition \ref{compirredfibzero}, it is enough to compute the dimension of $X_m$ for some $m$. We denote $P(m):=m(n-m)+n_1 m+n_2 (n-m)$ the dimension of $Z_m$. If $m \leq n_1$ or $m \geq n-n_2$, we have $\dim({X_m})=\dim(Z_m)=P(m)$.\\
If $n \geq n_1+n_2$, then $\dim(\NNN(W,G))=\dim(X_{n_1})=n n_1+ n n_2- n_1 n_2 . $\\
If $n < n_1+n_2$, then 
$$\dim(\NNN(W,G))=\dim \left(\bigcup_{i=\max(0,n-n_2)}^{\min(n,n_1)} X_i \right)= \max_{i={\max(0,n-n_2)}, \ldots,{\min(n,n_1)}} P(i)$$
and a simple study of the variations of the polynomial $P$ gives the result. 
\end{proof}

Now we are going to study the geometry of the fibers of $\nu$ over each orbit $U_i$. We recall that, by homogeneity, all the fibers over a given orbit are isomorphic. Hence, we just have to describe the fiber of $\nu$ over a point of each orbit. We fix bases $\BBB$, $\BBB_1$ and $\BBB_2$ of $V$, $V_1$ and $V_2$ respectively, and thus we can identify $W=\Hom(V_1,V) \times \Hom(V,V_2) \cong \MMM_{n,n_1} \times \MMM_{n_2,n}$ and $\Hom(V_1,V_2) \cong \MMM_{n_2,n_1}$, where $\MMM_{p,q}$ denotes the space of matrices of size $p \times q$.

Let $ \ 0 \leq r \leq N$, we denote 
\begin{equation} \label{defj}
J_r =\begin{bmatrix}
{I}_r \ \ &0_{r,n_1-r} \\
0_{n_2-r,r}  \ &0_{n_2-r,n_1-r} 
\end{bmatrix},
\end{equation}
where $I_r$ is the identity matrix of size $r$. The matrix $J_r$ identifies with an element of the orbit $U_r$ via the isomorphism $\Hom(V_1,V_2) \cong \MMM_{n_2,n_1}$.

We fix $r \in \{0, \ldots,N\}$, and we denote 
\begin{equation} \label{defwr}
w_r:=\left( \begin{bmatrix}
I_r &0 \\
0   &0 \end{bmatrix},\begin{bmatrix}
I_r  &0 \\
0  &0 \end{bmatrix} \right) \in W
\end{equation} 
and $G_r \subset G$ the stabilizer of ${w_r}$.

\begin{lemma} \label{orbfermee}
With the above notation, the orbit $G.{w_r} \subset W$ is closed in $W$, and is the unique closed orbit contained in ${\nu}^{-1}(J_r)$.
\end{lemma}

\begin{proof}
We have $\nu ({w_r})=J_r$ and it can be checked that 
$$G_r =\left\{ \begin{bmatrix}
I_r   &0 \\
0     &M \end{bmatrix},\ M\in {GL}_{n-r} \right\} \cong {GL}_{n-r};$$ 
this is a reductive subgroup of $G$. By \cite[\S I.6.2.5, Theorem 10]{SB}, we have the equivalence 
$$ G.{w_r} \text{ is closed in } W \Leftrightarrow C_G(G_r).{w_r} \text{ is closed in } W.$$ 
Then $C_G(G_r)=\left\{\begin{bmatrix}
M  &0 \\
0  &\lambda I_{n-r} \end{bmatrix},\ M \in {GL}_{r},\ \lambda \in \Gmm \right\}$, where $\Gmm$ denotes the multiplicative group. Hence 
$$C_G(G_r).{w_r}=\left\{ \left( \begin{bmatrix}
M   &0 \\
0   &0 \end{bmatrix},\begin{bmatrix}
M^{-1} &0 \\
0     &0 \end{bmatrix} \right)\ ,\ M \in {GL}_{r}\right\}$$ 
is a closed subset of $W$, and thus $G.{w_r}$ is a closed orbit in ${\nu}^{-1}(J_r)$. By \cite[\S II.3.1, Th\'{e}or\`{e}me 1]{SB}, the fiber ${\nu}^{-1}(J_r)$ contains a unique closed orbit, hence the result. 
\end{proof}

\begin{definition} \label{slice}
Let $x \in W$ such that the orbit $G.x$ is closed in $W$ and let $G_x \subset G$ be the stabilizer of $x$. By \cite[\S 6.2.1]{SB}, the $G_x$-module $T_x(G.x)$ admits a $G_x$-stable complement $M_x$ in $W$, which is called the \textit{slice representation} of $G_x$ at $x$. 
\end{definition}

\begin{lemma} \label{slice_explicite}
Let $r \in \{0,\ldots,N\}$, $G_r \subset G$ the stabilizer of $w_r$ defined by (\ref{defwr}), and $M_{w_r}$ the slice representation of $G_r$ at $w_r$. There is an isomorphism of $G_r$-modules
$$ M_{w_r} \cong  E_r^{\oplus n_1-r} \oplus  E_r^{* \oplus n_2-r} \oplus  V_0^{\oplus r(n_1+n_2-r)}, $$
where $E_r$ resp. $E_r^*$, is the defining representation resp. the dual representation, of $G_r$, and $V_0$ is the trivial representation of $G_r$.   
\end{lemma}

\begin{proof}
By definition of $M_{w_r}$, we have $M_{w_r} \cong W/T_{w_r}(G.{w_r})$ as $G_r$-modules. As $V \cong E_r \oplus V_{0}^{\oplus r}$ as $G_r$-modules, we deduce that 
\begin{equation*}
 W  \cong E_{r}^{\oplus n_1} \oplus E_{r}^{* \oplus n_2} \oplus V_{0}^{\oplus r(n_1+n_2)}.    
\end{equation*}
Then $T_{w_r}(G.{w_r}) \cong \ggg/\ggg_r$, where $\ggg$ resp. $\ggg_r$, denotes the Lie algebra of $G$ resp. of $G_r$. We have 
$$ \ggg \cong (V_{0}^{\oplus r}  \otimes V_{0}^{\oplus r} ) \oplus ( V_{0}^{\oplus r}  \otimes  E_r) \oplus (E_r^{*} \otimes V_{0}^{\oplus r} )  \oplus(E_r^{*} \otimes E_r) \ \text{ and } \ \ggg_r \cong E_r^{*} \otimes E_r.$$
Hence
$$  T_{w_r}(G.{w_r}) \cong E_{r}^{\oplus r} \oplus E_r^{* \oplus r} \oplus V_{0}^{\oplus r^2},$$
and thus
$$ M_{w_r} \cong  E_{r}^{\oplus n_1-r} \oplus E_r^{* \oplus n_2-r} \oplus V_{0}^{\oplus r(n_1+n_2-r)}$$  
as $G_r$-modules. 
\end{proof}

We denote ${\nu}_M:\ M_{w_r} \rightarrow M_{w_r}/\!/G_r$ the quotient morphism, and  
$$\NNN(N_{w_r},G_r):={{\nu}_M}^{-1}({\nu}_M(0))$$
the null cone of ${\nu}_M$. The group $G_r$ acts naturally on $G$ by right multiplication, as well as on $\NNN(M_{w_r},G_r)$ by definition of ${\nu}_M$.
We can thus consider the quotient
$$F_{w_r}:=G \times^{G_r} \NNN(M_{w_r},G_r), $$
which is naturally equipped with a $G$-scheme structure by \cite[\S I.5.14]{Ja}.\\ 
Next, we denote ${(W/\!/G)}^{(G_r)} \subset W/\!/G$ the subset of closed $G$-orbits of $W$ such that $G_r$ is conjugate to the stabilizer of a point of those orbits. In particular, Lemma \ref{orbfermee} implies that $G.{w_r} \in {(W/\!/G)}^{(G_r)}$. We denote $W^{(G_r)}:={\nu}^{-1}({(W/\!/G)}^{(G_r)}) \subset W$. Then, by \cite[\S 6.2.3, Theorem 8]{SB}, the sets ${(W/\!/G)}^{(G_r)}$ and $W^{(G_r)}$ are smooth subvarieties of $W/\!/G$ and $W$ respectively. Hence there is a morphism 
$${\nu '}:=\nu_{| W^{(G_r)}}:\ W^{(G_r)} \rightarrow {(W/\!/G)}^{(G_r)}$$ 
and, by \cite[\S 6.2.3, Theorem 8]{SB}, $\nu'$ is a fibration whose fiber is isomorphic to $F_{w_r}$. 
Hence 
$${\nu}^{-1}(J_r)= {\nu '}^{-1}(J_r) \cong G {\times}^{G_r} \NNN(M_{w_r},G_r) . $$
Let $F_1$, $F_2$ and $F_3$ be vector spaces of dimension $n_1-r$, $n_2-r$ and $r(n_1+n_2-r)$ respectively on which $G_r$ acts trivially. 
By Lemma \ref{slice_explicite}, there is an isomorphism of $G_r$-modules
$$ M_{w_r} \cong \Hom(F_1, E_r) \times \Hom(E_r, F_2) \times F_3  . $$ 
The quotient morphism ${\nu}_{M}$ is given by:
$$\begin{array}{lrcl}
 {\nu}_{M}:\  &\Hom(F_1, E_r) \times \Hom(E_r, F_2) \times F_3  & \rightarrow  & \Hom(F_1, F_2) \times F_3  . \\
        & (u_1',u_2',x)  & \mapsto      &  (u_2' \circ u_1',x) 
\end{array}$$
Hence $\NNN(M_{w_r},G_r):= {\nu}_{M}^{-1}({\nu}_M(0))={\nu}_{M}^{-1}(0) \cong {{\nu}_{M}'}^{-1}(0)$ with
\begin{equation} \label{qumro}
\begin{array}{lrcl}
 {\nu'_{M}}:\  &\Hom(F_1, E_r) \times \Hom(E_r, F_2) & \rightarrow  & \Hom(F_1, F_2)  .  \\
        & (u_1',u_2')  & \mapsto      &  u_2' \circ u_1' 
\end{array}
\end{equation}
The next proposition sums up our study of the fiber of $\nu$ over $J_r$ for $r=0, \ldots,N$. 

\begin{proposition} \label{descriptiongeofib}
Let $r \in \{0,\ldots,N\}$, $G_r \subset G$ be the stabilizer of $w_r$ defined by (\ref{defwr}), and ${\nu'_M}$ the quotient morphism defined by (\ref{qumro}). There is a $G$-equivariant isomorphism
$${\nu}^{-1}(J_r) \cong G {\times}^{G_r} {\nu'_M}^{-1}(0), $$
where $J_r \in U_r$ was defined by (\ref{defj}). In particular, if we denote $H:=G_{N}$ the stabilizer of $J_N$, we have 
$${\nu}^{-1}(J_{N}) \cong \left\{
    \begin{array}{ll}
        G &\text{ if } N=n; \\
        G/H &\text{ if } N=n_1=n_2<n; \\
        G {\times}^{H} \Hom(E_{N},F_2) &\text{ if } N=n_1<\min(n,n_2); \\
        G {\times}^{H} \Hom(F_1,E_{N}) &\text{ if } N=n_2<\min(n,n_1); 
    \end{array}
\right.$$ 
where $E_N$ is the defining representation of $H \cong GL_{n-N}$, $F_1$ and $F_2$ are vector spaces of dimension $n_1-N$ and $n_2-N$ respectively, and $H$ acts on $\Hom(F_1,E_{N}) \times \Hom(E_{N},F_2)$ by:
$$\forall h \in H,\ \forall (u_1',u_2') \in  \Hom(F_1,E_{N}) \times \Hom(E_{N},F_2),\ h.(u_1',u_2'):=(h \circ  u_1',u_2' \circ h^{-1}).$$    
\end{proposition}

\begin{corollary} \label{dimfibre}
Let $r \in \{0, \ldots,N\}$, the dimension of the fiber of $\nu$ over $J_r$ (defined by (\ref{defj})) is: 
\begin{itemize} \renewcommand{\labelitemi}{$\bullet$}
\item $n \, n_2+n \, r-n_2  r$ if $n-r \leq n_2-n_1$;
\item $n \, n_1+n \, r-n_1 r$  if $n-r \leq n_1-n_2$;
\item $\frac{1}{2}(n-r)(n_1+n_2)+\frac{1}{4}{(r+n)}^2+\frac{1}{4}{(n_1-n_2)}^2$ if $|n_1-n_2|<n-r<n_1+n_2-2r$ and $n+n_1-n_2-r$ is even; 
\item $\frac{1}{2}(n-r)(n_1+n_2)+\frac{1}{4}{(r+n)}^2+\frac{1}{4}{(n_1-n_2)}^2-\frac{1}{4}$ if $|n_1-n_2|<n-r<n_1+n_2-2r$ and $n+n_1-n_2-r$ is odd; 
\item $n n_1+n n_2 -n_1 n_2$ if $n \geq n_1+n_2-r$.
\end{itemize}
\end{corollary}

\begin{proof}
By Proposition \ref{descriptiongeofib}, we have
$$\dim({\nu}^{-1}(J_r))=\dim(G {\times}^{G_r} {\nu'_M}^{-1}(0))=2nr-r^2+\dim({\nu'_M}^{-1}(0)),$$
and Corollary \ref{fibzero} gives $\dim({\nu'_M}^{-1}(0))$ in terms of $n$, $n_1$, $n_2$ and $r$.  
\end{proof}

For each triple $(n,n_1,n_2)$, Corollary \ref{dimfibre} allows one to compute the dimension of the general fibers of $\nu$, and also to determine the flat locus of $\nu$. The proof of the following proposition is analogous to Proposition \ref{ouvert_platitudeSLn}. 

\begin{proposition} \label{ouvertplatitude}
The dimension of the general fibers and the flat locus of $\nu$ are given by the following table:
\vspace*{0.5mm}
\begin{center}
\begin{tabular}{|c|c|c|}
  \hline
  configuration & dim. of the general fibers & flat locus \\
  \hline
  $n>\max(n_1,n_2)$ &     $n n_1+n n_2-n_1 n_2$ & $U_{N} \cup  \cdots \cup U_{\max(n_1+n_2-n-1, 0)}$\\
  $n=\max(n_1,n_2)$ &     $n^2$ & $U_{N} \cup U_{N-1} $\\  
  $\min(n_1,n_2) \leq n<\max(n_1,n_2)$     & $n n_1+n n_2-n_1 n_2$ & $U_{N}$\\
  $n < \min(n_1,n_2)$  & $n^2$ & $U_{N}$\\ 
  \hline    
\end{tabular}\\
\end{center}
\vspace*{0.5mm}
\end{proposition}

The following corollary is a consequence of Propositions \ref{chow} and \ref{ouvertplatitude}.

\begin{corollary} \label{cas_facile}
The morphism $\nu$ is flat if and only if $n \geq n_1+n_2-1$, in which case $\HHH \cong W/\!/G=\Hom(V_1,V_2)$ and the Hilbert-Chow morphism $\gamma$ is an isomorphism. 
\end{corollary}

In the next proposition, we determine, for each $M \in \Irr(G)$, the multiplicity of the $G$-module $M$ in $k[{\nu}^{-1}(J_{N})]$. 

\begin{proposition} \label{fcthilb}
We denote $H:=G_{N} \cong GL_{n-N}$ the stabilizer of $w_{N}$ defined by (\ref{defwr}), $E_N$ resp. $E_N^*$, the defining representation of $H$ resp. the dual representation of $H$, and $F_1$, $F_2$ two vector spaces of dimension $n_1-N$ and $n_2-N$ respectively on which $H$ acts trivially. The Hilbert function $h_W$ of the general fibers of $\nu$ is given by:\\
$\forall M \in \Irr(G),\ h_W(M)= \left\{
    \begin{array}{ll}
        \dim(M) &\text{ if } N=n; \\
        \dim(M^{H}) &\text{ if } N=n_1=n_2<n; \\
        \dim({(M {\otimes} k[\Hom(E_{N},F_2)])}^{{H}}) &\text{ if } N=n_1<\min(n,n_2); \\
        \dim({(M {\otimes} k[\Hom(F_1,E_{N})])}^{{H}}) &\text{ if } N=n_2<\min(n,n_1). 
    \end{array}
\right.$
\end{proposition}

\begin{proof}
We use the description of the fiber of $\nu$ over $U_N$ from Proposition \ref{descriptiongeofib}.\\ 
$\bullet$ If $N=n$, then 
$$k[{\nu}^{-1}(J_{N})] \cong k[G] \cong \bigoplus_{M \in \Irr(G)} M^{*} \otimes M$$ 
as $G \times G$-modules, and thus $h_W(M)=\dim(M)$.\\
$\bullet$ If $N=n_1=n_2<n$, then 
$$k[{\nu}^{-1}(J_{N})] \cong k[G/H] \cong {k[G]}^{H} \cong \bigoplus_{M \in \Irr(G)} M^{*H} \otimes M$$ 
as a left $G$-module, and thus $h_W(M)=\dim(M^{*H})=\dim(M^H)$ as $H$ is reductive.\\
$\bullet$ If $N=n_1<\min(n_2,n)$, then
$$k[{\nu}^{-1}(J_{N})] \cong k[G {\times}^{H} \Hom(E_{N},F_2)] \cong \bigoplus_{M \in \Irr(G)} M^{*} \otimes { \left(M {\otimes} k[\Hom(E_{N},F_2)] \right)}^{{H}},$$
and thus $h_W(M)=\dim \left({ \left(M {\otimes} k[\Hom(E_{N},F_2)]\right)}^H \right)$.\\
The case $N=n_2<\min(n_1,n)$ is similar to the previous case. 
\end{proof}

\subsection{Description of the coordinate ring of the null cone}  \label{subsectionJ}

The aim of this section is to show Corollary \ref{decompoiso} that gives a description of the coordinate ring of the null cone as a $G' \times G$-module for $n_1=n_2=n$. It will be enough for our purpose to consider only this particular case because in Section \ref{reducGLn} we will obtain the reduction principle for $GL_n$ that will allow to reduce the case $n_1,n_2 \geq n$ to $n_1=n_2=n$. Our reference for the representation theory of classical groups is \cite{FH}. We fix once and for all a Borel subgroup $B' \subset G'$, a maximal torus $T' \subset B'$, and we denote $U'$ the unipotent radical of $B'$. In the same way, and with obvious notation, we fix subgroups $T,\ B$ and $U$ of $G$. 

We denote $J$ the ideal of $k[W]$ generated by the homogeneous $G$-invariants of positive degree. The ideal $J$ is $G' \times G$-stable by definition. We denote
\begin{equation} \label{defslV}
\ssl(V):=\{ f \in \End(V)\ \mid  \ \tra(f)=0 \},
\end{equation} 
where $\tra(f)$ is the trace of the endomorphism $f$. We have $V^* \otimes V \cong \End(V) \cong \ssl(V) \oplus V_0$ and
\begin{align}
{k[W]}_2 \cong &(S^2(V_1) \otimes S^2(V^{*})) \oplus (S^2(V) \otimes S^2(V_{2}^{*}))  \oplus  ({\Lambda}^2 (V_1) \otimes {\Lambda}^2 (V^{*}))  \label{kW2} \\
       &\ \oplus  ({\Lambda}^2 (V) \otimes {\Lambda}^2 (V_{2}^{*})) \oplus (V_1 \otimes V_{2}^{*} \otimes \ssl(V)) \oplus (V_1 \otimes V_{2}^{*} \otimes V_0) \notag
\end{align}
as a $G' \times G$-module. Hence $J \cap {k[W]}_2 = V_1 \otimes V_{2}^{*} \otimes V_0 \cong \Hom(V_2,V_1)$ as a $G' \times G$-module, and this module generates the ideal $J$. We recall that we fixed bases $\BBB$, $\BBB_1$ and $\BBB_2$ of $V$, $V_1$ and $V_2$ respectively. The following result is due to Kraft and Schwarz (\cite[\S 9]{KS}):

\begin{proposition} \label{KSheadings}  
With the above notation, let $x_i$, $1 \leq i \leq n$, be the $i$-th principal minor of $\Hom(V_1,V) \cong \MMM_{n,n_1}$ and $y_j$, $1 \leq j \leq n$, the $j$-th "antiprincipal" minor (that is, the minors starting from the lower right corner) of $\Hom(V,V_2)\cong \MMM_{n_2,n}$. Then the $T' \times T$-algebra ${(k[W]/J)}^{U' \times U}$ is generated by the $x_i$ and the $y_j$ and the relations between these generators are generated by the monomials $\{x_i y_j \mid  i+j>n\}$. In other words, there is an exact sequence 
$$0 \rightarrow J' \rightarrow k[x_1, \ldots,x_n,y_1, \ldots,y_n] \rightarrow {(k[W]/J)}^{U' \times U} \rightarrow 0,$$
where $J'$ is the monomial ideal generated by $\{x_i y_j\ \mid \ i+j>n\}$. 
\end{proposition}

We denote $\Lambda$ the weight lattice of $G$ and $\Lambda_+ \subset \Lambda$ the subset of dominant weights. If $\lambda \in \Lambda_+$, we denote $S^{\lambda}(V)$ the irreducible $G$-module of highest weight $\lambda$. For an appropriate choice of a basis $\{ \epsilon_1,\ldots,\epsilon_n\}$ of $\Lambda$, we have $\lambda=r_1 \epsilon_1+\ldots+r_n \epsilon_n \in \Lambda_+$ if and only if $r_1 \geq r_2 \geq \ldots \geq r_n$. If $\lambda=r_1 \epsilon_1+\ldots+r_n \epsilon_n \in \Lambda_+$, we define $\lambda^*=-r_n \epsilon_1-r_{n-1} \epsilon_2-\ldots-r_1 \epsilon_n \in \Lambda_+$. We then have $S^{\lambda^*}(V) \cong S^{\lambda}(V^*)$. Furthermore, we say that $\lambda \geq 0$ if $r_n \geq 0$. One can then write any $\lambda \in \Lambda_+$ uniquely in the form $\lambda = \alpha + \beta^*$, where $\alpha,\ \beta \in \Lambda_+$ and $\alpha,\ \beta \geq 0$.

\begin{corollary} \label{decompoiso}
With the above notation, let $\lambda=\alpha+\beta^* \in \Lambda_+$. Then the isotypic component associated to the $G$-module $S^{\lambda}(V)$ in $k[W]/J$ is 
$$S^\beta(V_1) \otimes S^\alpha (V_2^*) \otimes S^\lambda(V).$$
In addition, the representation $S^{\lambda}(V)$ appears in $k[W]_p/(J \cap k[W]_p)$ if and only if $p=\sum_i |r_i|$.   
\end{corollary}

\begin{proof}
By Proposition \ref{KSheadings}, there is an isomorphism of $T' \times T$-algebras
$${(k[W]/J)}^{U' \times U} \cong k[x_1, \ldots,x_n,y_1, \ldots,y_n]/J'.$$
We denote $k_i:=|r_i|$, and let $1 \leq t \leq n$ be the integer such that $\alpha=r_1 \epsilon_1+\ldots+r_t \epsilon_t$. One may check that the weight of the monomial 
$$  x_{n-t}^{k_{t+1}} x_{n-t-1}^{k_{t+2}-k_{t+1}} x_{n-t-2}^{k_{t+3}-k_{t+2}}  \cdots x_{1}^{k_{n}-k_{n-1}} y_{t}^{k_{t}} y_{t-1}^{k_{t-1}-k_{t}} y_{t-2}^{k_{t-2}-k_{t-1}}  \cdots y_{1}^{k_{1}-k_{2}}$$ 
is $(\lambda,\beta,\alpha^*)$ and that $\lambda$ uniquely determines this monomial. It follows that the isotypic component of the $G$-module $S^{\lambda}(V)$ in $k[W]/J$ is $S^\beta(V_1) \otimes S^\alpha (V_2^*) \otimes S^\lambda(V)$.
Hence, the representation $S^{\lambda}(V)$ appears in $k[W]_p/(J \cap k[W]_p)$ if and only if 
\begin{align*}
p&=(k_n-k_{n-1})+2(k_{n-1}-k_{n-2})+\ldots+(n-t)k_{t+1}+(k_1-k_2)+2(k_2-k_3)+\ldots+t k_t\\
 &=k_1+k_2+ \ldots+k_n.
\end{align*}  
\end{proof}

\begin{remark} 
By Corollary \ref{decompoiso}, if $M$ is a polynomial representation, or the dual of a polynomial representation, then  the multiplicity of $M$ in $k[W]/J$ equals $\dim(M)$. 
\end{remark}

\subsection{The reduction principle for $GL_n$}  \label{reducGLn}

From now on, we suppose that $n_1,n_2 \geq n$. We fix $(E_1,E_2) \in \Gra(n,V_1^*) \times \Gra(n,V_2)$, and let $P$ be the stabilizer of $(E_1,E_2)$ for the natural action of $G'$ on $\Gra(n,V_1^*) \times \Gra(n,V_2)$. We denote 
$$W':=\Hom(V_1/E_1^{\perp},V) \times \Hom(V,E_2)$$ 
and 
$$\HHH':=\Hilb_{h_{W'}}^{G}(W').$$ 
We will obtain in Lemma \ref{fiberGLn} that $\HHH'$ identifies naturally with a $P$-stable closed subscheme of $\HHH$. Our aim is to show the reduction principle in Case 4 stated in the introduction. More precisely, we will show:

\begin{proposition} \label{reduction2}
With the above notation, there is a $G'$-equivariant isomorphism 
$$\begin{array}{ccccc}
\phi & : & G' {\times}^{P} \HHH' & \cong & \HHH.  \\
& & (g',[Z])P & \mapsto & g'.[Z] 
\end{array}$$
\end{proposition}

First, we need the following lemma, whose proof is analogous to that of Lemma \ref{exi1SLn}, in order to apply the Key-Proposition to $M=V$ and $M=V^*$.

\begin{lemma} \label{exi1}
Let $G=GL(V)$ and $W=\Hom(V_1,V) \times \Hom(V,V_2)$, with $n_1,n_2 \geq n$. Then the $k[W]^G$-module ${k[W]}_{(V)}=\Hom^G(V,k[W])$ resp. ${k[W]}_{(V^{*})}=\Hom^G(V^*,k[W])$, is generated by ${\Hom}^{G}(V,W^{*})$ resp. by ${\Hom}^{G}(V^*,W^{*})$.
\end{lemma}

We have ${\Hom}^{G}(V,W^{*}) \cong V_2^*$ and ${\Hom}^{G}(V^*,W^{*}) \cong V_1$ as $G'$-modules. By Proposition \ref{fcthilb}, we have $h_W(V)=h_W(V^*)=n$. Therefore, the Key-Proposition gives two $G'$-equivariant morphisms $\delta_{V^*}:\ \HHH \rightarrow \Gra(n,V_1^*)$ and $\delta_V:\ \HHH \rightarrow \Gra(n,V_2)$. Identifying $\Gra(n,V_1^*) \times \Gra(n,V_2) \cong G'/P$ and then taking the product of $\delta_{V^*}$ and $\delta_{V}$ gives a $G'$-equivariant morphism
\begin{equation*}
 \rho:\ \HHH \rightarrow G'/P,
\end{equation*} 
whence a $G'$-equivariant isomorphism
\begin{equation*}
\HHH \cong G' \times^P F,
\end{equation*}
where $F$ is the scheme-theoretic fiber of $\rho$ over $eP$. The next lemma, whose proof is analogous to Lemma \ref{fibrehil}, completes the proof of Proposition \ref{reduction2}.

\begin{lemma} \label{fiberGLn}
With the above notation, there is a $P$-equivariant isomorphism 
$$F \cong \HHH',$$
where $P$ acts on $\HHH'$ via its action on $W'$. 
\end{lemma}

\begin{corollary} \label{reductionHdiag}
With the above notation, there is a commutative diagram
\begin{equation*} 
 \xymatrix{
     G' {\times}^{P} \HHH' \ar[d]_{\phi} \ar[rr]^{G' \times^P \gamma'} && G' {\times}^{P} W'/\!/G \ar[d]^{\theta} \\
    \HHH \ar[rr]_{\gamma}    &&  W/\!/G 
    }
  \end{equation*}
where $\phi$ is the isomorphism of Proposition \ref{reduction2}, $\gamma':\ \HHH' \rightarrow W'/\!/G$ is the Hilbert-Chow morphism, $G' \times^P \gamma'$ is the morphism induced by $\gamma'$, and $\theta$ is the morphism induced by the inclusion $W'/\!/G \subset W/\!/G$.   
\end{corollary}

\begin{proof}
It follows from Proposition \ref{reduction2} that the diagram
\begin{equation*} 
   \xymatrix{
    \HHH'  \ar[r] \ar[d]_{\gamma'}  & \HHH \ar[d]^{\gamma} \\
    W'/\!/G \ar[r]_{\theta} & W/\!/G
  }  
\end{equation*}
is commutative, where the upper horizontal arrow is the closed embedding of Lemma \ref{fiberGLn}. Then the following diagram
\begin{equation*} 
 \xymatrix{
     G' {\times}^{P} \HHH' \ar[d]_{G' \times^P \gamma'} \ar[r] & G' {\times}^{P} \HHH \ar[d]^{G' \times^P \gamma} \ar[r]^{\cong} & G'/P \times \HHH \ar[d]^{Id \times \gamma} \\
     G' {\times}^{P} W'/\!/G  \ar[r] & G' {\times}^{P} W/\!/G  \ar[r]_(0.4){\cong} & G'/P \times W/\!/G  }
  \end{equation*}   
is also commutative, hence the result. 
\end{proof}

\subsection{The case $n_1=n_2=2$}  \label{etudeGL2}

The aim of this section is to determine $\HHH$ in Case 4 for $n_1=n_2=n=2$. Next, in Section \ref{masterpropositioncasn2}, we will use the reduction principle for $GL_n$ to deduce the case where $n_1,n_2 \geq n=2$.

\subsubsection{Fixed-points of $\HHH$ for the action of $B'$}  \label{subsectionptfixe}

First, we want to show that $\HHH$ is a smooth variety. We recall that we fixed a Borel subgroup $B' \subset G'$. By Lemma \ref{Hlisse4}, it is enough to show that the fixed-points of $\HHH$ for the action of $B'$ are contained in the main component $\HHHp$; and then to check that the dimension of the tangent space to $\HHH$ in each of those fixed-points equals the dimension of $\HHHp$. Hence, we have to determine the fixed-points of $B'$ in $\HHH$. 

We denote $D$ the unique $B'$-stable line of $V_1 \otimes V_{2}^{*}$, and $I$ the ideal of $k[W]$ generated by $(V_1 \otimes V_{2}^{*} \otimes V_0) \oplus (D \otimes \ssl(V)) \subset {k[W]}_2$. The ideal $I$ is homogeneous, $B' \times G$-stable and contains the ideal $J$ generated by the homogeneous $G$-invariants of positive degree in $k[W]$.

\begin{theorem} \label{pointfixeborel}
The ideal $I$ defined above is the unique fixed-point of $\HHH$ for the action of the Borel subgroup $B'$.
\end{theorem}

\begin{proof}
Let $[Z]$ be a closed point of $\HHH$, and $I_Z \subset k[W]$ the ideal of the $G$-stable closed subscheme $Z \subset W$. The point $[Z]$ is fixed for the action of $B'$ if and only if $I_Z$ is $B'$-stable. In particular, $I_Z$ is stable by the group of invertible scalar matrices, and thus $I_Z$ has to be homogeneous.  
Hence, the fixed-points of $\HHH$ for the action of $B'$ correspond exactly to the homogeneous ideals $I_Z$ of $k[W]$ such that:\\
\indent i) $I_Z$ is $B' \times G$-stable; and\\
\indent ii) $k[W]/ I_Z \cong \bigoplus_{M \in \Irr(G)} M^{\oplus \dim(M)}$ as a $G$-module.\\
As the algebra $k[W]$ is graded and as the ideal $I_Z$ is homogeneous, the algebra $k[W]/I_Z$ is also graded:
\begin{equation*}
k[W]/I_Z= \bigoplus_{p \geq 0} {k[W]}_p/(I_Z \cap {k[W]}_p).
\end{equation*} 
We can thus study $I_Z$ degree by degree. It is clear that ${k[W]}_0 \cap I_Z=\{0\}$ and that ${k[W]}_1 \cap I_Z \neq k[W]_1$. Let us now study the component of degree 2 using the decomposition (\ref{kW2}). To get the decomposition given by ii), we must have ${k[W]}_2 \cap I_Z \supseteq \ssl(V) \oplus V_0^{\oplus 4}$. Indeed, the $G$-module $k[W]/I_Z$ already contains a copy of the trivial representation given by the constants, hence $k[W]/I_Z$ cannot contain any other copy. Then, ${k[W]}_2$ contains four copies of $\ssl(V)$ which is a 3-dimensional $G$-module, hence ${k[W]}_2 \cap I_Z$ contains at least one copy of $\ssl(V)$. As $k[W]_2 \cap I_Z$ is $B'$-stable, $k[W]_2 \cap I_Z$ contains $D \otimes \ssl(V)$ because $D$ is the unique $B'$-stable line of $V_1 \otimes V_{2}^{*}$. It follows that $I_Z$ contains $(V_1 \otimes V_{2}^{*} \otimes V_0) \oplus (D \otimes \ssl(V))$, and therefore $I_Z \supset I$.  
The next lemma implies that this inclusion is in fact an equality, which completes the proof of Theorem \ref{pointfixeborel}. 

\begin{lemma} \label{l3generators}
The Hilbert function of the ideal $I$ defined aboved is $h_W$, the Hilbert function of the general fibers of $\nu:\ W \rightarrow W/\!/G$.
\end{lemma}

\textit{Proof of the lemma:}
We recall that, for each $M \in \Irr(G)$, we have $h_W(M)=\dim(M)$, and thus we have to show that
$$k[W]/I \cong \bigoplus_{M \in \Irr(G)} M^{\oplus \dim(M)}$$  
as a $G$-module. There is an inclusion of ideals $J \subset I$, whence the exact sequence of $B' \times G$-modules
$$ 0 \rightarrow I/J \rightarrow k[W]/J \rightarrow k[W]/I  \rightarrow 0, $$ 
hence 
$$k[W]/I \cong \frac{k[W]/J}{I/J}$$ 
as $B' \times G$-modules. 
Corollary \ref{decompoiso} provides the decomposition of $k[W]/J$ into irreducible $G' \times G$-modules. If $M$ is a polynomial $G$-module or its dual, then the multiplicity of $M$ in $k[W]/J$ is $\dim(M)$. Otherwise, $M = S^{k_1{\epsilon}_1-k_2 {\epsilon}_2}(V)$ for a unique couple $k_1,k_2 >0$ and the multiplicity of $S^{k_1{\epsilon}_1-k_2 {\epsilon}_2}(V)$ in $k[W]/J$ equals $\dim(S^{k_2}(V_1) \otimes S^{k_1}(V_{2}^{*}))={(k_1+1)(k_2+1)}$.\\ 
Recalling that we identify $W=\Hom(V_1,V) \times \Hom(V,V_2)$ with $\MMM_{2,2} \times \MMM_{2,2}$, we write 
\begin{equation*} 
w=\left( \begin{bmatrix}
x_{11}  &x_{12} \\
x_{21}  & x_{22} 
\end{bmatrix},
\begin{bmatrix}
 y_{22} & y_{12} \\
 y_{21} & y_{11}
\end{bmatrix} \right) \in W,
\end{equation*}
and we identify $k[W]$ with the algebra $k[x_{ij},y_{ij},\ 1 \leq i,j \leq 2]$. By Proposition \ref{KSheadings}, we have 
$${(k[W]/J)}^U \cong k[x_{11},x_{12},x_{11} x_{22}-x_{21} x_{12},y_{11},y_{12},y_{11} y_{22}-y_{21} y_{12}]/K,$$
where
\begin{align*} 
K=(&x_{11} (y_{11} y_{22}-y_{21} y_{12}), x_{12} (y_{11} y_{22}-y_{21} y_{12}), y_{11} (x_{11} x_{22}-x_{21} x_{12}),\\
   &y_{12} (x_{11} x_{22}-x_{21} x_{12}),(x_{11} x_{22}-x_{21} x_{12})(y_{11} y_{22}-y_{21} y_{12})).
\end{align*}
The ideal ${(I/J)}^U$ of ${(k[W]/J)}^U$ is generated by $x_{11}y_{11}$. Hence     
$$I/J \cong \bigoplus_{k_1,k_2 >0} {\left(D.\left(S^{k_2-1}(V_1) \otimes S^{k_1-1}(V_{2}^{*})\right)\right)} \otimes S^{k_1{\epsilon}_1-k_2 {\epsilon}_2}(V)$$
as a $B' \times G$-module. It follows that
\begin{align*} 
k[W]/I \cong &V_0 \oplus {k[V_{1}^{*} \otimes V]}_{+} \oplus {k[V^{*} \otimes V_2]}_{+} \\
           &\oplus  \left(\bigoplus_{k_1,k_2 >0 } \frac{S^{k_2}(V_1) \otimes S^{k_1}(V_{2}^{*})}{(D.(S^{k_2-1}(V_1) \otimes S^{k_1-1}(V_{2}^{*})))} \otimes S^{k_1{\epsilon}_1-k_2 {\epsilon}_2}(V)\right)
\end{align*} 
as a $B' \times G$-module, where ${k[V_{1}^{*} \otimes V]}_{+}$ resp. ${k[V^{*} \otimes V_2]}_{+}$, denotes the maximal homogeneous ideal of $k[V_{1}^{*} \otimes V]$ resp. of $k[V^{*} \otimes V_2]$. \\
We notice that the multiplicities of the polynomial representations in $k[W]/I$, and those of their duals, are the same as in $k[W]/J$. However, the multiplicity of 
$S^{k_1{\epsilon}_1-k_2 {\epsilon}_2}(V)$ in $k[W]/I$ is
$$\dim \left(\frac{S^{k_2}(V_1) \otimes S^{k_1}(V_{2}^{*})}{(D.(S^{k_2-1}(V_1) \otimes S^{k_1-1}(V_{2}^{*})))}\right)=(k_1+1)(k_2+1)-k_1 k_2=\dim(S^{k_1{\epsilon}_1-k_2 {\epsilon}_2}(V)).$$ 
\end{proof}

\begin{remark}
We have $\Stab_{G'}(I)=B'$, hence the unique closed orbit of $\HHH$ is isomorphic to $G'/B' \cong \PPr^1 \times \PPr^1$.
\end{remark}

The next corollary is a direct consequence of Lemma \ref{fixespoints} and Theorem \ref{pointfixeborel}.

\begin{corollary} \label{Hconnexe}
The scheme $\HHH$ is connected.
\end{corollary}

\subsubsection{Tangent space to $\HHH$ at $[Z_0]$} \label{dimTang}
Let $[Z_0]$ be the unique fixed point of $\HHH$ for the action of $B'$. We will show:

\begin{proposition} \label{dimTangent}
The dimension of the Zariski tangent space to $\HHH$ at the unique fixed-point $[Z_0]$ for the action of the Borel subgroup $B'$ is $\dim(T_{[Z_0]} \HHH)=4$.
\end{proposition}

We identify $k[W]$ with the algebra $k[x_{ij},y_{ij},\ 1 \leq i,j \leq 2]$ as in the proof of Lemma \ref{l3generators} and we fix explicit bases for some $B' \times G$-modules appearing in $k[W]_2$: \\

 $\left.
    \begin{array}{l}
      f_1:=y_{22} x_{11}+y_{12} x_{21}\\
      f_2:=y_{22} x_{12}+y_{12} x_{22}\\
      f_3:=y_{21} x_{11}+y_{11} x_{21}\\
      f_4:=y_{21} x_{12}+y_{11} x_{22}
      \end{array}
\right \}  \text{ is a basis of } V_1 \otimes V_2^* \otimes V_0; $\\

$\left.
    \begin{array}{l}
         h_1:=x_{11} y_{11}\\
         h_2:=x_{11} y_{21}\\  
         h_3:=x_{21} y_{11}\\ 
         h_4:=x_{21} y_{21}
      \end{array}
\right \} \text{ is a basis of }  D \otimes V^* \otimes V \cong D \otimes (\ssl(V) \oplus V_0).$\\

The ideal $I$ defined at the beginning of Section \ref{subsectionptfixe} is the ideal of the closed subscheme $Z_0 \subset W$, and we denote $R:=k[W]/I$ the algebra of global sections of the structure sheaf of $Z_0$. We consider
\begin{equation*}
N_1:=\left\langle  f_1,f_2,f_3,f_4,h_1,h_2,h_3,h_4 \right\rangle \subset k[W],
\end{equation*}
which is a $B' \times G$-submodule that generates the ideal $I$. One can check that the generators of $I$ satisfy the following relations: 
\begin{equation*}
 \left \{
    \begin{array}{l}
   r_1:=-f_1 \otimes y_{21}+h_2 \otimes y_{22}+ h_4 \otimes y_{12};\\
   r_2:=-f_3 \otimes x_{11}+h_1\otimes x_{21}+ h_2 \otimes x_{11}; \\
   r_3:=f_4 \otimes x_{11}-h_1 \otimes x_{22}- h_2\otimes x_{12}. 
      \end{array}
      \right.
 \end{equation*}
By Lemma \ref{fixespoints}, the variety $\HHHp$ contains at least one fixed-point for the action of $B'$, and thus $[Z_0] \in \HHHp$. It follows that $\dim(T_{[Z_0]} \HHH) \geq \dim(\HHHp)=4$. On the other hand, by Lemma \ref{InegTang}, we have $\dim(T_{[Z_0]} \HHH)=7-\rank(\rho^*)$, where $\rho^*$ is the morphism of Diagram (\ref{complexeCornFlakes2}). Hence, to show Proposition \ref{dimTangent}, it is enough to prove:

\begin{lemma} 
With the above notation, we have $\rank(\rho^*) \geq 3.$
\end{lemma}

\begin{proof}
For $i=1, \ldots ,4$, we define ${\psi}_i \in {\Hom}_R^G(R \otimes N_1,R)$ by: 
$$\left\{
    \begin{array}{ll}
        {\psi}_i( h_j \otimes 1)=0 &\text{ for $j=1,\ldots,4$}; \\
        {\psi}_i(f_j \otimes 1)={\delta}_i^j &\text{ for $j=1,\ldots,4$};
    \end{array}
\right.
$$
where ${\delta}_i^j$ is the Kronecker symbol whose value is $1$ if $i=j$, and $0$ otherwise. The ${\psi}_i$ are linearly independent in $\Hom_R^G(R {\otimes} N_1,R)$, and we are going to show that $\rho^*({\psi}_1)$, $\rho^*({\psi}_3)$ and $\rho^*({\psi}_4)$ are linearly independent in $\Ima(\rho^*)$, which will prove the lemma. Let ${\lambda}_1$, ${\lambda}_3$, ${\lambda}_4 \in k$ such that 
\begin{equation} \label{relLinaire1}
\sum_{i=1,3,4} {\lambda}_i \, \rho^*({\psi}_i)=0.
\end{equation}
Evaluating (\ref{relLinaire1}) at $r_1 \otimes 1$, we get
$$0={\sum}_{i=1,3,4} {\lambda}_i \, \rho^*({\psi}_i)(r_1 \otimes 1)= {\sum}_{i=1,3,4} {\lambda}_i \, {\psi}_i(r_1) = -{\lambda}_1 \, y_{21}.$$
Then, evaluating (\ref{relLinaire1}) at $r_2\otimes 1$ and $r_3\otimes 1$, we get
$-{\lambda}_3 \, {x_{11}}={\lambda}_4 \, x_{11}=0.$\\
We deduce that $({\lambda}_1, {\lambda}_3, {\lambda}_4)=(0,0,0)$, hence the result.  
\end{proof}

The next corollary follows from Lemma \ref{Hlisse4} and Proposition \ref{dimTangent}.

\begin{corollary} \label{Hcasn2}
$\HHH=\HHHp$ is a smooth variety. 
\end{corollary}

\subsubsection{Construction of a $G'$-equivariant morphism $\delta:\ \HHH \rightarrow\PPr(\Hom(V_1,V_2))$}  \label{subsectionconstructionmor}  

The next lemma follows from classical invariant theory and can be shown in the same way as Lemma \ref{exi1SLn}.

\begin{lemma} \label{lemmeThClass2}
Let $G=GL(V)$ and $W=\Hom(V_1,V) \times \Hom(V,V_2)$, with $n_1=n_2=n=2$. Then the ${k[W]}^G$-module ${k[W]}_{(\ssl(V))}=\Hom^G(\ssl(V),k[W])$ is generated by ${\Hom}^{G}(\ssl(V),k[W]_2)$, where the $G$-module $\ssl(V)$ is defined by (\ref{defslV}). 
\end{lemma}

There is an isomorphism of $G'$-modules
$${\Hom}^{G}(\ssl(V),k[W]_2) \cong V_1 \otimes V_2^*.$$
Therefore, the Key-Proposition gives a $G'$-equivariant morphism $\HHH \rightarrow \PPr(V_1 \otimes V_2^*)$. But $V_1 \otimes V_{2}^{*} \cong \det(V_1) \otimes V_{1}^{*} \otimes V_{2} \otimes {\det}^{-1}(V_2)$ as a $G'$-module, and thus
\begin{equation*} 
\PPr (V_1 \otimes V_{2}^{*}) \cong \PPr (\det(V_1) \otimes V_{1}^{*} \otimes V_{2} \otimes {\det}^{-1}(V_2)) \cong \PPr (V_{1}^{*} \otimes V_{2}) \cong \PPr (\Hom(V_1,V_2)) 
\end{equation*}
as $G'$-varieties. Hence, we get a $G'$-equivariant morphism
\begin{equation} \label{maoprhismedelta}
\delta :\ \HHH \rightarrow  \PPr (\Hom(V_1,V_2)). 
\end{equation}

\subsubsection{The morphism $\gamma \times \delta$ is an isomorphism between $\HHH$ and $\YYY_0$}  \label{subsectionison2}

We recall that we defined $\YYY_0$ at the beginning of Section \ref{yy}. The action of $G'$ on $\Hom(V_1,V_2)$ induces an action of $G'$ on $\Hom(V_1,V_2) \times \PPr(\Hom(V_1,V_2))$ that preserves $\YYY_0$. Proceeding as for Lemma \ref{morphisme_dans_YSLn}, one may check that the morphism $\gamma \times \delta$ sends $\HHH$ into $\YYY_0$.

\begin{lemma} \label{bijens}
Let $\delta$ be the $G'$-equivariant morphism defined by (\ref{maoprhismedelta}), and $\YYY_0$ be the variety defined at the beginning of Section \ref{yy}. Then, the morphism $\gamma \times \delta:\ {\HHH} \rightarrow \YYY_0$ is quasi-finite.
\end{lemma}

\begin{proof}                                                                                                                                                      As before, we denote $[Z_0] \in \HHH$ the unique fixed-point for the action of the Borel subgroup $B'$. Let $Y_0:=(\gamma \times \delta)([Z_0])$, and let $[Z] \in \HHH$ such that $(\gamma \times \delta)([Z])=Y_0$. We have $\gamma([Z])=0$, hence the ideal $I_Z \subset k[W]$ of $Z$ contains $V_1 \otimes V_2^* \otimes V_0$; and we have $\delta([Z])=\delta([Z_0])$, hence $I_Z$ contains $D \otimes \ssl(V) \subset k[W]_2$, where $D$ is the unique $B'$-stable line of $V_1 \otimes V_2^*$. By Lemma \ref{l3generators}, we have $Z=Z_0$, and thus $(\gamma \times \delta)^{-1}(Y_0)$ is a finite set.
Then, we consider
$$E:=\{[Z] \in \HHH \text{ such that the fiber over } Y:=({\gamma \times \delta})([Z]) \text{ is of dimension } \geq 1  \},$$
which is a $G'$-stable closed subset of $\HHH$ by \cite[Exercise II.3.22]{Ha}. By Lemma \ref{fixespoints}, if $E$ is non-empty then $E$ has to contain $[Z_0]$, but we just have seen that this is not the case, hence the result.  
\end{proof}

\begin{proposition} \label{gammaiso}
With the notation of Lemma \ref{bijens}, $\gamma \times \delta:\ \HHH  \rightarrow  \YYY_0$ is an isomorphism.
\end{proposition}

\begin{proof}
The variety $ \YYY_0$ is smooth, hence normal, and $\gamma \times \delta:\ {\HHH} \rightarrow  \YYY_0$ is birational, proper and quasi-finite by Lemma \ref{bijens}. Therefore, by Zariski's Main Theorem, $\gamma \times \delta$ is an isomorphism.  
\end{proof}

\subsection{The case $n_1,n_2 \geq 2$}     \label{masterpropositioncasn2}

In this section, we suppose that $n_1,n_2 \geq n=2$, and thus $W/\!/G=\Hom(V_1,V_2)^{\leq 2}$. Our aim is to treat the last case of Theorem \ref{GL2} using the reduction principle for $GL_2$.

First of all, we introduce some notation: For any $L \in \PPr(\Hom(V_1,V_2)^{\leq 2})$, we denote $\Ker(L):=\Ker(l)$, $\Ima(L):=\Ima(l)$ and $\rank(L):=\rank(l)$, where $l$ is any nonzero element of $L$.\\
If $\rank(L)=2$, we denote $L \wedge L:=\left\langle l \wedge l \right\rangle \in \PPr(\Hom(\Lambda^2 (V_1),\Lambda^2 (V_2))^{\leq 1})$.\\
Finally, we denote $\iota_1:\ \Gra(2,V_1^*) \rightarrow \PPr(\Lambda^2(V_1^*))$ and $\iota_2:\ \Gra(2,V_2) \rightarrow \PPr(\Lambda^2(V_2))$ the Pl\"{u}cker embeddings of $\Gra(2,V_1^*)$ and $\Gra(2,V_2)$ respectively.

Let us consider the variety
\begin{align*}
\ZZZ:=\{&(f,L,E_1,E_2) \in \Hom(V_1,V_2)^{\leq 2} \times \PPr(\Hom(V_1,V_2)^{\leq 2}) \times \Gra(2,V_1^*) \times \Gra(2,V_2) \ | \\
            & f \in L, E_1^{\perp} \subset \Ker(L) \text{ and } \Ima(L) \subset E_2 \},
\end{align*}
then there is a diagram
\begin{equation*}
\xymatrix{ &  \ZZZ \ar@{->>}[ld]_{q_1} \ar@{->>}[rd]^{q_2} \\   \YYY_0 && \Gra(2,V_1^*) \times \Gra(2,V_2) }
\end{equation*}
where $q_1$ and $q_2$ are the natural projections. We recall that the varieties $\YYY_0$ and $\YYY_1$ were defined in Section \ref{yy}.

\begin{lemma} \label{identificationblowup}
With the previous notation, there is an isomorphism
$\YYY_1 \cong \ZZZ$
which identifies the blow-up $\YYY_1 \rightarrow \YYY_0$ with $q_1:\ \ZZZ \rightarrow \YYY_0$.
\end{lemma}

\begin{proof}
We denote $\FFF_0$ the strict transform of $\Hom(V_1,V_2)^{\leq 1}$ via the blow-up $\YYY_0 \rightarrow \Hom(V_1,V_2)^{\leq 2}$, and we consider 
$$\alpha: \newdir{ >}{{}*!/-5pt/\dir{>}}
  \xymatrix{
     \YYY_0 \ar@{-->}[r]  & \PPr(\Hom(\Lambda^2(V_1),\Lambda^2(V_2))^{\leq 1}) 
    }$$ 
the rational map defined over
$$U:=\{(f,L) \in \YYY_0\ \mid  \rank(L)=2\}=\YYY_0 \backslash \FFF_0$$ 
by $\alpha(f,L)=L \wedge L$. Let $\Gamma \subset  \YYY_0 \times \PPr(\Hom(\Lambda^2(V_1),\Lambda^2(V_2))^{\leq 1})$ be the graph of $\alpha$, then
$$\Gamma=\{(f,L,E)\in U \times \PPr(\Hom(\Lambda^2(V_1),\Lambda^2(V_2))^{\leq 1}) \mid  L \wedge L =E \}.$$
The closure of $\Gamma$ inside $\YYY_0 \times \PPr(\Hom(\Lambda^2(V_1),\Lambda^2(V_2))^{\leq 1})$ is
$$\overline{\Gamma}=\left \{(f,L,E)\in \YYY_0 \times \PPr(\Hom(\Lambda^2(V_1),\Lambda^2(V_2))^{\leq 1})\ \middle|  \begin{array}{l}
\forall l  \in L,\ l \wedge l  \in E; \\
\Ima(E) \in \iota_2(\Gra(2,V_2));\\
\Ker(E)^{\perp} \in \iota_1(\Gra(2,V_1^*)).
\end{array} \right \}, $$
where $\Ima(E):=\Ima(g)$ and $\Ker(E):=\Ker(g)$ for any non-zero element $g \in E$.\\

\noindent \underline{Claim:} The map $\overline{\Gamma} \rightarrow \YYY_0,\ (f,L,E) \mapsto (f,L)$ is the blow-up of $\YYY_0$ along $\FFF_0$. \\

To prove the claim, we start by constructing a covering of $\YYY_0$ by affine open sets. Recalling that $\PPr(\Hom(V_1,V_2)^{\leq 2}) \subset \PPr(\Hom(V_1,V_2)) \cong \PPr(\MMM_{n_2,n_1})$, we define an affine open subset $O_{i,j} \subset \PPr(\MMM_{n_2,n_1})$, for all $1 \leq i \leq n_2$ and $1 \leq j \leq n_1$, by fixing the $(i,j)$-coordinate equal to 1. Then, we denote 
$$\YYY_0^{i,j}:=\{(f,l) \in \Hom(V_1,V_2)^{\leq 2} \times (\PPr(\Hom(V_1,V_2)^{\leq 2}) \cap O_{i,j})\ |\ f \in \left\langle l \right\rangle\},$$ 
which is an affine open subset, maybe empty, of $\YYY_0$. The $\YYY_0^{i,j}$ form a covering of $\YYY_0$. Let $\III$ be the ideal sheaf of $\FFF_0$ in $\YYY_0$, then $\III_{| \YYY_0^{i,j}}$ is the ideal of $k[\YYY_0^{i,j}]$ generated by the minors of size 2 of the matrix $l$. The restriction of $\alpha$ to $\YYY_0^{i,j}$ is the rational map $\alpha_{i,j}: \newdir{ >}{{}*!/-5pt/\dir{>}}
  \xymatrix{
\YYY_0^{i,j} \ar@{-->}[r]  & \PPr(\Hom(\Lambda^2(V_1),\Lambda^2(V_2))^{\leq 1})} $
defined over $U_{i,j}:=\{(f,l) \in \YYY_0^{i,j} \ |\ \rank(l)=2\}$ by $\alpha_{i,j}(f,l)=\left\langle  l \wedge l \right\rangle$.          
We denote $\Gamma_{i,j} \subset \YYY_0^{i,j} \times \PPr(\Hom(\Lambda^2(V_1),\Lambda^2(V_2))^{\leq 1})$ the graph of $\alpha_{i,j}$, then by \cite[Proposition IV.22]{EH}, the morphism $\overline{\Gamma_{i,j}} \rightarrow \YYY_0^{i,j},\ (f,l,E) \mapsto (f,l)$ is the blow-up of $\YYY_0^{i,j}$ along $\FFF_0 \cap \YYY_0^{i,j}$. One can check that these blows-up glue together to obtain the morphism $\overline{\Gamma} \rightarrow \YYY_0,\ (f,L,E) \mapsto (f,L)$ as expected.\\

It remains to show that $\overline{\Gamma}$ is isomorphic to $\ZZZ$. The Segre embedding gives an isomorphism
$$\begin{array}{ccc}
\PPr(\Hom(\Lambda^2(V_1),\Lambda^2(V_2))^{\leq 1}) & \cong & \PPr(\Lambda^2(V_1)^*) \times \PPr(\Lambda^2(V_2)), \\
 E & \mapsto & (\Ker(E)^{\perp}, \Ima(E)) 
\end{array}$$
hence an isomorphism
$$\overline{\Gamma} \cong \left \{ (f,L,L_1,L_2)\in \YYY_0 \times \PPr(\Lambda^2(V_1)^*) \times \PPr(\Lambda^2(V_2))\ \middle| \begin{array}{l}
       L_1 \in \iota_1(\Gra(2,V_1^*)); \\
          L_2 \in \iota_2(\Gra(2,V_2));\\
         \forall l \in L, \   \Ker(l \wedge l)^{\perp} \subset L_1; \\        
          \hspace{12mm} \Ima(l \wedge l)  \subset L_2. 
\end{array} \right \}.$$
Let $f \in \Hom(V_1,V_2)^{\leq 2}$, then $f \wedge f \in \Hom(\Lambda^2(V_1), \Lambda^2(V_2))^{\leq 1}$. If $\rank(f) \leq 1$, then $f \wedge f=0$. Otherwise, $\Ima(f \wedge f)=L_2$ for some $L_2 \in \iota_2(\Gra(2,V_2))$, and $\Ker(f \wedge f)^{\perp}=L_1$ for some $L_1 \in \iota_1(\Gra(2,V_1^*))$. For $i=1,2$, we denote $E_i$ the preimage of $L_i$ by $\iota_i$. We then have the equivalences
\begin{align*}
\Ima(f \wedge f)=L_2 &\Leftrightarrow \Ima(f)=E_2; \text{ and } \\  
\Ker(f \wedge f)^{\perp}=L_1 &\Leftrightarrow \Ker(f)^{\perp}=E_1. 
\end{align*}
It follows that $\overline{\Gamma} \cong \ZZZ$.   
\end{proof}

We identify $\YYY_1$ with $\ZZZ$ by Lemma \ref{identificationblowup}. The natural action of $G'$ on 
$$ \Hom(V_1,V_2)^{\leq 2} \times \PPr(\Hom(V_1,V_2)^{\leq 2}) \times \Gra(2,V_1^*) \times \Gra(2,V_2)$$ 
stabilizes $\YYY_1$. We fix $(E_1,E_2) \in \Gra(2,V_1^*) \times \Gra(2,V_2)$, and we denote $P \subset G'$ the stabilizer of $(E_1,E_2)$. Identifying $\Gra(2,V_1^*) \times \Gra(2,V_2) \cong G'/P$, and denoting $F'$ the scheme-theoretic fiber of the natural projection $\YYY_1 \rightarrow G'/P$, there is a $G'$-equivariant isomorphism    
\begin{equation*}
\YYY_1 \cong G' {\times}^{P} F'.
\end{equation*}
One can check that 
$$F' \cong \{(f,L) \in \Hom(V_1/E_1^{\perp},E_2) \times \PPr( \Hom(V_1/E_1^{\perp},E_2))\ |\ f \in L\}$$
is the blow-up of $\Hom(V_1/E_1^{\perp},E_2)$ at $0$. In other words, $\YYY_1$ is the total space of the $G'$-homogeneous bundle $Bl_0(\Hom(\underline{V_1}/T_1^{\perp},T_2))$, where $T_1$ resp. $T_2$, is the tautological bundle of $\Gra(2,V_1^*)$ resp. of $\Gra(2,V_2)$, and $\underline{V_1}$ is the trivial bundle with fiber $V_1$ over $\Gra(2,V_1^*)$. By Propositions \ref{reduction1} and \ref{gammaiso}, there is a $G'$-equivariant isomorphism 
\begin{equation*}
\HHH \cong \YYY_1.
\end{equation*}

With the notation of Corollary \ref{reductionHdiag}, there is a commutative diagram      
\begin{equation*} 
\xymatrix {
    {G' {\times}^P {\HHH}'} \ar[rr]^{G' \times^P {\gamma}'} \ar[rd]_{\cong} \ar[dd]_{\phi}^{\cong} && {G' {\times}^P W'/\!/G} \ar[dd]^{\theta} \\
    & {G' {\times}^P Bl_0(W'/\!/G)} \ar[ru]_{pr_{1,2}} \ar[dd]_(0.3){\cong} \\
    \HHH \ar[rr]_(0.6){\gamma}  \ar[rd]_{\cong} && {W/\!/G} \\
    & \YYY_1 \ar[ru]_{pr_1} }  
\end{equation*}   
where $Bl_0(W'/\!/G)$ denotes the blow-up of $W'/\!/G=\Hom(V_1/E_1^{\perp},E_2)$ at 0, 
$$\begin{array}{ccccc}
pr_{1,2} & : & G' {\times}^P Bl_0(W'/\!/G) & \rightarrow & G' {\times}^P W'/\!/G, \\
& & (g',f,L)P & \mapsto & (g',f)P 
\end{array}$$
and
$$\begin{array}{ccccc}
pr_1 & : & \YYY_1 & \rightarrow &  W/\!/G.  \\
& & (f,L,E_1,E_2) & \mapsto & f 
\end{array}$$  
In particular, the Hilbert-Chow morphism $\gamma$ identifies with the composition of the blows-up $\YYY_1 \rightarrow \YYY_0 \rightarrow W/\!/G$, which completes the proof of Theorem \ref{GL2}.

\begin{proposition}
Let $G=GL(V)$ and $W=\Hom(V_1,V) \times \Hom(V,V_2)$, with $n_1=n_2>n=2$. Then, $W/\!/G=\Hom(V_1,V_2)^{\leq 2}$ is singular and Gorenstein. Moreover, the desingularization $\gamma:\ \HHH \rightarrow W/\!/G$ is not crepant.
\end{proposition}

\begin{proof}
We saw in Section \ref{description_quotient} that $W/\!/G=\Hom(V_1,V_2)^{\leq 2}$ is singular and Gorenstein when $n_1=n_2>n=2$. There are two well-known crepant desingularizations of $W/\!/G$:
\begin{equation*}
\xymatrix{ &  R_1 \ar[ld]_{p_{12}} \ar[rd]^{p_{11}} && R_2 \ar[ld]_{p_{21}} \ar[rd]^{p_{22}} & \\   
            \Gra(n,V_1^*) &&   W/\!/G && \Gra(n,V_2)  }
\end{equation*}
where
\begin{align*}
R_1:&=\left \{(M,L_1) \in W/\!/G \times \Gra(n,V_1^*) \ \mid  \ L_1^{\perp} \subset \Ker(M) \right \};\\
R_2:&=\left \{(M,L_2) \in W/\!/G \times \Gra(n,V_2) \ \mid  \ \Ima(M) \subset L_2 \right \};
\end{align*}
and the $p_{ij}$ are the natural projections. The variety $R_1$ resp. $R_2$, is the total space of the vector bundle $\Hom(\underline{V_1}/T_1^{\perp},\underline{V_2}) \cong T_1 \otimes \underline{V_2}$ over $\Gra(n,V_1^*)$ resp. of the vector bundle $\Hom(\underline{V_1},T_2) \cong \underline{V_1^*} \otimes T_2$ over $\Gra(n,V_2)$, hence $R_1$ resp. $R_2$, is smooth and $p_{11}$ resp. $p_{21}$, is a desingularization of $W/\!/G$. Furthermore, the exceptional locus $X_1:=p_{11}^{-1}(\overline{U_{n-1}})$ of $p_{11}$ resp. $X_2:=p_{21}^{-1}(\overline{U_{n-1}})$ of $p_{21}$, is of codimension $n_2-n+1 \geq 2$ resp. $n_1-n+1\geq 2$, hence $p_{11}$ and $p_{21}$ are crepant.\\
Let us now consider the fibered product $R:=R_1 \times_{W/\!/G} R_2$. One can check that $R$ is isomorphic to the total space  of the vector bundle $\Hom(\underline{V_1}/T_1^{\perp},T_2) \cong T_1 \times T_2$ over $\Gra(n,V_1^*) \times \Gra(n,V_2)$, hence $R$ is a smooth variety. By Theorem \ref{GL2}, there is a commutative diagram 
\begin{equation*}
\xymatrix{\HHH \ar[rd]^{q} \ar@/_7pc/[rddd]_(0.4){\gamma} && \\
 &R \ar@{->>}[dd]^{\phi}\ar@{->>}[ld]_{\phi_1} \ar@{->>}[rd]^{\phi_2} & \\
 R_1 \ar[rd]_{p_{11}} && R_2 \ar[ld]^{p_{21}} \\
 & W/\!/G &
}
\end{equation*}
where $\phi_1$ and $\phi_2$ are the natural projections, and $q$ is the blow-up of $\Hom(\underline{V_1}/T_2,T_1)$ along the zero-section. But crepant desingularizations are minimal (see \cite[Lemme A.1.2]{Te1} for a proof of this fact), and thus, as $\HHH$ is not isomorphic to $R_1$ or $R_2$, the desingularization $\gamma$ cannot be crepant. 
\end{proof}

\ \\

\noindent \texttt{Université Grenoble I, Institut Fourier,}\\
\texttt{UMR 5582 CNRS-UJF, BP 74,}\\
\texttt{38402 St. Martin d'Hères Cédex, FRANCE}\\
\textit{E-mail address:} \texttt{ronan.terpereau@ujf-grenoble.fr}


\begin{thebibliography}{ASM}

 \bibitem[AB]{AB} V. Alexeev, V.  and M. Brion, 
{\it Moduli of affine schemes with reductive group action},
J. Algebraic Geom. {\bf 14} (2005), 83--117.

 \bibitem[Be]{Be} T. Becker, 
{\it An example of an {$SL_2$}-{H}ilbert scheme with multiplicities},
{T}ransform. {G}roups {\bf 16} (2011), no. 4, 915--938.

\bibitem[Bo]{Bo}  A. Borel, 
{\it Linear Algebraic Groups (second edition)},
Graduate Texts in Mathematics {\bf 126}, Springer-Verlag, New York, 1991.

 \bibitem[BKR]{BKR} T. Bridgeland, A. King, and M. Reid,
{ \it The {M}ac{K}ay correspondence as an equivalence of derived categories},
J. Amer. Math. Soc.  {\bf 14} (2001), 535--554.

\bibitem[Br]{Br} M. Brion, 
{\it Invariant {H}ilbert schemes},
Handbook of {M}oduli: {V}olume I, {A}dvanced {L}ectures in {M}athematics {\bf 24}, Fordham {U}niversity, {N}ew {Y}ork, 2013, 63--118.

\bibitem[Ei]{Ei} D. Eisenbud,
{\it Commutative Algebra},
Graduate Texts in Mathematics {\bf 150}, Springer-Verlag, New York, 1995.

\bibitem[EH]{EH} D. Eisenbud and J. Harris,
{\it The Geometry of Schemes},
Graduate Texts in Mathematics {\bf 197}, Springer-Verlag, New York, 2001.

\bibitem[FH]{FH} W. Fulton and J. Harris,
{\it Representation Theory},
Graduate Texts in Mathematics {\bf 129}, Springer-Verlag, New York, 1991.


\bibitem[HS]{HS}  M. Haiman and B. Sturmfels,
{\it Multigraded {H}ilbert schemes},
J. Algebraic Geom. {\bf 13} (2004), 725--769.

\bibitem[Ha]{Ha}  R. Hartshorne, 
{\it Algebraic Geometry},
Graduate Texts in Mathematics {\bf 52}, Springer-Verlag, New York, 1977.

 \bibitem[IN1]{IN1} Y. Ito and I. Nakamura,
{\it Mc{K}ay correspondence and {H}ilbert schemes},
Proc. Japan Acad. Ser. A Math. Sci. {\bf 72} (1996), no. 7, 135--138.

 \bibitem[IN2]{IN2} Y. Ito and I. Nakamura,
{\it Hilbert schemes and simple singularities},
\emph{in:} New trends in algebraic geometry ({W}arwick, 1996), 151--233, {L}ondon Math. Soc. Lecture Note Ser. {\bf 264}, Cambridge Univ. Press, 1999.

 \bibitem[Ja]{Ja} J. Jantzen,
{\it Representations of Algebraic Groups (second edition)},
Mathematical Surveys and Monographs {\bf 107}, AMS, Providence, RI, 2003.

 \bibitem[Kr]{Kr} H. Kraft, 
{\it {G}eometrische {M}ethoden in der {I}nvariantentheorie},
{A}spects of {M}athematics, {D}1., {F}riedr. {V}ieweg and {S}ohn, {B}raunschweig, 1984.

 \bibitem[KS]{KS} H. Kraft  and G. W. Schwarz,
{\it Representations with a reduced null cone},
arXiv: 1112.3634, to appear in {P}rogress in {M}athematics ({B}irkh\"{a}user), a volume in honor of {N}olan {W}allach.

 \bibitem[LS]{LS} M. Lehn and C. Sorger, 
{\it A symplectic resolution for the binary tetrahedral group},
S\'{e}minaires et {C}ongres {\bf 24}-{II} (2010), 427--433.

 \bibitem[Pr]{Pr}  C. Procesi, 
{\it Lie Groups, an Approach through Invariants and Representations},
Universitext, Springer, New York, 2007.

 \bibitem[SB]{SB} G. W. Schwarz and M. Brion, 
{\it Th\'{e}orie des invariants et g\'{e}om\'{e}trie des vari\'{e}t\'{e}s quotients},
Travaux en cours {\bf 61}, Hermann, Paris, 2000.

 \bibitem[Sv]{Sv} T. Svanes, 
 {\it {C}oherent cohomology on {S}chubert subschemes of flag schemes and applications},
Advances in {M}ath. {\bf 14} (1974), 369--453.

\bibitem[Te1]{Te1} R. Terpereau, 
{\it Sch\'{e}mas de {H}ilbert invariants et th\'{e}orie classique des invariants ({P}h.{D}. thesis)},
arXiv: 1211.1472.

\bibitem[Te2]{Te2} R. Terpereau, 
{\it Invariant {H}ilbert schemes and desingularizations of symplectic reductions for classical groups},
arXiv: 1303.3032, , to appear in Math. Z.

 \bibitem[We]{We} J. Weyman, 
{\it Cohomology of vector bundles and {S}yzygies},
Cambridge Tracts in Mathematics {\bf 149}, Cambridge University Press, 2003.

\end{thebibliography}
\end{document}